\documentclass[11pt, a4paper, oneside, reqno]{amsart}
\usepackage{amsmath}

\usepackage[numbered]{bookmark}
\usepackage{changepage} 
\usepackage[utf8]{inputenc}
\usepackage{graphics, graphicx}
\usepackage{amsmath, amsfonts, amsthm, amssymb}
\usepackage{mathrsfs, mathtools, mathabx}
\usepackage{verbatim} 
\usepackage{enumerate}
\usepackage{xcolor}
\usepackage{cite}
\usepackage{array}

\usepackage{longtable}
\usepackage[pagewise]{lineno}

\newtheorem{theorem}{Theorem}[section]
\newtheorem{definition}[theorem]{Definition}
\newtheorem{lemma}[theorem]{Lemma}
\newtheorem{proposition}[theorem]{Proposition}
\newtheorem{corollary}[theorem]{Corollary}

\newtheorem*{claim}{Claim}

\newcommand{\A}{\mathcal{A}} 
\newcommand{\orth}[1]{#1^\bot}
\newcommand{\HC}[2]{{}^TQ_{#1, #2}}
\newcommand{\HCT}[2]{Q_{#1, #2}}

\newcommand{\Kpi}{\textup{Ker}(\Omega_\pi)}

\newcommand{\cZ}{\mathcal{Z}}

\newcommand{\Dom}{X_{\mathcal A}}
\newcommand{\DomE}{\mathfrak{X}_{\mathcal A}}

\newcommand{\cZE}{{\mathcal{Z}_{\textup{ext}}}}

\newcommand{\Z}{\mathbb{Z}} 
\newcommand{\N}{\mathbb{N}}

\newcommand{\R}{\mathbb{R}}
 
\newcommand{\T}{\mathbb{T}}

\newcommand{\FirstColumn}{0.35\textwidth}

\newcommand{\Function}[5]{\begin{array}{llll} #1 : & #2 & \rightarrow & #3 \\ & #4 & \mapsto & #5 \end{array}}
\date{}
\author{Przemysław Berk, Frank Trujillo}
\begin{document}
\renewcommand{\arraystretch}{1.2}
\sloppy
\title{Rigidity for piecewise smooth circle homeomorphisms and certain GIETs}
\maketitle
\begin{abstract}
In this article, we prove a rigidity property for a class of generalized interval exchange transformations (GIETs), which contains the class of piecewise smooth circle homeomorphisms.   

More precisely, we show that if two piecewise $C^3$ GIETs $f$ and $g$ with zero mean non-linearity are topologically conjugated, boundary-equivalent, have the same typical combinatorial rotation number and their renormalizations approach in an appropriate way the set of affine interval exchange transformations, then, their respective renormalizations converge to each other exponentially and the conjugating map is of class $C^1$.  In addition, if $f$ and $g$ are GIETs with rotation-type combinatorial data (i.e., they naturally define a piecewise smooth circle homeomorphism), have the same typical combinatorial rotation number, and they are break-equivalent as piecewise smooth circle homeomorphisms, then they are $C^1$-conjugated as circle maps. 

This work provides the first rigidity results for GIETs not smoothly conjugated to IETs, and generalizes a previous result of K. Cunha and D. Smania \cite{cunha_rigidity_2014}, concerning an exceptional class of circle maps, in the setting of piecewise $C^3$ circle homeomorphisms.
\end{abstract}
\section{Introduction}
\emph{Generalized interval exchange transformations} (GIETs) are piecewise smooth bijective maps of a compact interval with a finite number of discontinuities and non-negative derivative. They appear naturally as first return maps to Poincaré sections of smooth flows on surfaces.

In this work, we investigate the rigidity properties of GIETs, that is, the 
relation between the map's topological and smooth conjugacy classes. In other 
words, given two GIETs $f$ and $g$, which are topologically conjugated, we 
study the smoothness of this conjugacy.  {When the topological conjugacy class 
of a given map coincides with its smooth conjugacy class (or with an 
appropriate subset of it after adding trivial restrictions for the existence of 
a smooth conjugacy), we speak of \emph{rigidity}.  

We will show that rigidity is a typical phenomenon among  \emph{irrational} GIETs with \emph{zero mean non-linearity} (see Section \ref{sc:typical} for a discussion on the notion of typical used in this work and Sections \ref{sc:notations}, \ref{sc:GIETs} for definitions). For the sake of clarity of exposition, we postpone a rigorous statement of our main result (Theorem \ref{thm: rigidity}) to Section \ref{sc: results}.} 

Similar to the classical Poincaré classification theorem for circle maps, {an 
irrational GIET is semi-conjugated to a \emph{standard interval exchange 
transformation} (IET), that is, to a GIET that is a piecewise translation. 
Moreover, this semi-conjugacy is a conjugacy if and only if the GIET is minimal 
(see \cite[Proposition 7]{yoccoz_echanges_2005}).} 


The smooth conjugacy class of IETs in the space of GIETs was first studied by 
S. Marmi, P. Moussa, and J.C. Yoccoz in  \cite{marmi_linearization_2012}, where 
the authors show that, for a typical {minimal} IET $T$, the set of GIETs 
smoothly conjugated to $T$ defines a smooth submanifold of positive finite 
codimension (in the space of $C^{r}$ deformations of $T$ which are tangent to 
$T$ at the singularities). A generalization of this result for an exceptional 
class of IETs 
has been obtained by S. Ghazouani in \cite{ghazouani_local_2021}.  More recently, S. Ghazouani and C. Ulcigrai \cite{ghazouani_priori_2023} provided sufficient conditions for a GIET to be smoothly conjugated to an IET.

Rigidity is an important aspect in many dynamical systems, and such properties have been extensively studied in several settings closely related to GIETs. By a classical result of A. Denjoy, sufficiently smooth circle homeomorphisms with irrational rotation number are topologically conjugated to rotations. In this case, the regularity of the conjugacy map has been thoroughly studied by several authors \cite{herman_sur_1979, yoccoz_conjugaison_1984, katznelson_differentiability_1989, sinai_smoothness_1989, khanin_hermans_2009}. These results showed a crucial interplay between the regularity of the transformation being considered and the arithmetic properties of its rotation number. 

Circle homeomorphisms with singularities, i.e., (piecewise) smooth circle maps whose derivative either vanishes at a finite number of points (critical circle maps) or is discontinuous at a finite number of points while admitting non-vanishing left and right derivatives (circle diffeomorphisms with breaks), are, under quite general assumptions, topologically conjugated to circle rotations. Indeed, it follows from a result by J.C. Yoccoz \cite{yoccoz_il_1984} that critical circle maps with irrational rotation number and non-flat critical points are topologically conjugated to irrational rotations, while Denjoy's result remains valid in the case of singularities of break-type (see  \cite[Chapter VI]{herman_sur_1979}). However, the rigidity question is slightly different in this setting. In fact, we cannot expect a circle map with singularities to be smoothly conjugated to a rotation due to the presence of zeroes or jump discontinuities in the derivative of the map. Still, it is natural to ask about the regularity of the conjugacy between two circle homeomorphisms with the same rotation number and type of singularities. 

Rigidity for critical circle maps has been shown for sufficiently regular circle maps of irrational rotation number having exactly one singularity of odd criticality, see, e.g.,\cite{de_faria_rigidity_1999, de_faria_rigidity_2000, khanin_robust_2007, guarino_rigidity_2017, guarino_rigidity_2018}.  Similar rigidity results hold for circle diffeomorphisms with exactly one break, see, e.g., \cite{khanin_robust_2007, teplinskyi_smooth_2010, khanin_renormalization_2014, kocic_generic_2016, khanin_robust_2018, akhadkulov_rigidity_2021}. However, in this case, rigidity holds only for almost every rotation number and cannot be extended to every irrational rotation number (see \cite{khanin_absence_2013}).

The setting of critical circle maps with multiple singularities has yet to be 
explored more thoroughly. Still, some first results (see, e.g., the works of G. 
Estevez, E. de Faria, P. Guarino, D. Smania and M. Yampolsky  
\cite{estevez_real_2018, estevez_beau_2018, estevez_hyperbolicity_2021, 
estevez_renormalization_2022, estevez_renormalization_2023}) point to the 
existence of rigidity phenomena. Rigidity for $C^3$ critical circle maps with 
irrational rotation number and a finite number of singularities has been 
formally conjectured by G. Estevez, and E. de Faria in 
\cite{estevez_real_2018}, 

Concerning circle diffeomorphisms with several breaks, rigidity results have been recently obtained by K. Cunha, and D. Smania \cite{cunha_rigidity_2014} for an exceptional class of maps.

{Let us point out that many of the results mentioned above rely on 
\emph{renormalization} -- a powerful tool in studying numerous low-dimensional 
dynamical systems. Besides the previously cited examples, some notable classes 
that have been studied through renormalization are those of unimodal (see, 
e.g., \cite{feigenbaum_quantitative_1978, coullet_iterations_1978, 
feigenbaum_quasiperiodicity_1982, mcmullen_complex_1994, 
lyubich_feigenbaum-coullet-tresser_1999}) and Lorenz maps (see. e.g., 
\cite{winckler_renormalization_2010, martens_hyperbolicity_2014}).

}

In this article, we provide the first rigidity results for GIETs not smoothly 
conjugated to IETs and generalize the rigidity results in 
\cite{cunha_rigidity_2014} to typical circle diffeomorphisms with breaks of 
class $C^3$. 
{Before describing in detail these results, let us say a few words about the 
idea of typicality used throughout this work.

\subsection{A notion of typicality for {infinitely renormalizable} GIETs.}
\label{sc:typical}
The notion of \emph{typical} used in this work relies on the renormalization 
theory of GIETs and their \emph{combinatorial rotation number} (see Section 
\ref{sc:notations}), which can be seen as an analog of the notion of rotation 
number of circle homeomorphisms in the GIETs setting. {This notion means that 
the maps are in some sense `combinatorially typical', and it is inspired by the 
works of S. Marmi, P. Moussa, and J.C.~Yoccoz \cite{marmi_affine_2010, 
marmi_linearization_2012} where the authors show rigidity and existence of 
wandering intervals, respectively, for certain classes of combinatorially 
typical GIETs}. We say that a property $P$ holds for typical maps in a given 
subset of {infinitely renormalizable} GIETs (e.g., $C^3$ on $d$ intervals) if 
there exists a full-measure set of combinatorial rotation numbers 
$\mathfrak{C}$ (see Section \ref{sc:notations} for a definition of the measure 
on the space of combinatorial rotation numbers) such that any map in this class 
whose combinatorial rotation number belongs to $\mathfrak{C}$ satisfies the 
property $P$. As a slight abuse of terminology, we sometimes simply say that a 
typical map in the given class verifies the property $P$. Also, we might refer 
simply to `a typical map' without explicitly mentioning the full-measure set to 
which its combinatorial rotation number belongs. {We refer the reader to 
Section \ref{sc:notations} for additional details on combinatorial rotation 
numbers and the notion of tipicality}.

As mentioned before, it follows from  \cite[Proposition 
7]{yoccoz_echanges_2005} that irrational GIETs are semi-conjugated to IETs. As 
we shall see in Section \ref{sc:notations}, typical {infinitely renormalizable} 
GIETs are irrational, and the measure on the set of combinatorial rotation 
numbers is nothing more than the push-forward of the `Lebesgue measure' (more 
precisely, the product of Lebesgue and counting measures) on the set of IETs to 
the space of combinatorial rotation numbers. This will allow us to restate the 
notion of typicality above as follows: A property $P$ holds for typical maps in 
a given subset of {infinitely renormalizable} GIETs (on $d$ intervals) if there 
exists a full-measure set of IETs $\mathcal{C}$ (on $d$ intervals) such that 
any map in this class that is semi-conjugated to an IET in $\mathcal{C}$ 
satisfies the property $P$. For the sake of clarity, we write the statements of 
the results in this work in this way. 

Notice that this notion of typicality extends immediately to the class of circle diffeomorphisms with breaks. Indeed, piecewise smooth circle homeomorphisms can be naturally identified with GIETs of \emph{rotation type} (see \eqref{eq: rotation_type} for a precise definition and Section \ref{sc: p-homeomorphisms} for details on this identification).

{Let us point out that how this notion of typicality reflects on the 
(infinitely dimensional) parameter space of GIETs (see Section 
\ref{sc:notations}) is not well understood. In fact, although our definition of 
typical follows by analogy the case of circle diffeomorphisms and their 
classification according to their rotation number, in contrast to the latter 
case, when considering a finite-dimensional parameter family of GIETs, the 
measure of the parameters corresponding to infinitely renormalizable GIETs is 
expected to have measure zero. This indicates that the function associating to 
each map its rotation number is less regular and thus, although one can choose 
this family so that a positive measure set of combinatorial rotation numbers is 
realized (see, e.g., \cite{marchese_full_2018}), it is not clear whether 
imposing a full-measure condition on the combinatorial rotation number 
corresponds to `most' of the parameters corresponding to infinitely 
renormalizable GIETs, for example, to a set of parameters with full Hausdorff 
dimension. Moreover, this situation is not exclusive to GIETs, but it appears 
in the exact same way when considering the parameters corresponding to maps 
with irrational rotation number in finite-dimensional families of circle maps 
with singularities. 

Understanding how this or other notions of typical maps reflect on the parameter spaces of GIETs and circle maps with singularities is an interesting question in itself, as it is not completely clear which notion should be used to study the class of infinitely renormalizable maps, which, as mentioned before, is expected to be a very small subset in the parameter space and of zero measure when restricted to most finite-dimensional parameter families.  However, this question is out of the scope of this work, and we content ourselves with following a more standard approach to describe the maps according to their combinatorial rotation number, in the spirit of \cite{marmi_affine_2010, marmi_linearization_2012}.} 



}

\subsection{Description of results}
\label{sc:description_results}

{This work concerns} GIETs with zero mean non-linearity satisfying 
the \emph{EC Condition} (see Definition \ref{def:EC}), where the latter 
assumption describes the way the map transforms under renormalization. { We 
point out that this property is known to hold for both typical $C^3$ circle 
diffeomorphisms with breaks and typical $C^3$ minimal GIETs with 4 or 5 
intervals when they have zero mean non-linearity. Indeed, this follows from 
recent results by S. Ghazouani and C. Ulcigrai in \cite{ghazouani_priori_2023}. 
More precisely, by Corollary 4.5.1 and Lemmas 4.6.1, 4.6.2 in 
\cite{ghazouani_priori_2023} typical GIETs with zero mean non-linearity having 
\emph{a priori bounds} verify the EC Condition (we refer the reader to 
\cite{ghazouani_priori_2023} for a precise definition of a priori bounds and to 
Sections 4.5  and 4.6 therein for details on the proof of this fact). Moreover, 
by Denjoy's inequality (see \cite[Proposition 7.4.4]{herman_sur_1979}) 
piecewise $C^2$ circle diffeomorphisms with breaks have a priori bounds and by 
\cite[Theorem D]{ghazouani_priori_2023} typical GIETs with 4 or 5 intervals 
have a priori bounds.}  
 
The main result of this article is Theorem \ref{thm: rigidity}, which states that for two typical topologically conjugated GIETs $f$ and $g$ of class $C^3$ with zero mean non-linearity, having the same \emph{boundary} (see Section \ref{sc:GIETs}), and satisfying the EC Condition, the conjugating map is of class $C^1$. This result is stated in a slightly more general fashion as the assumption of $C^0$ conjugacy can be replaced by having the same combinatorial rotation number. Indeed, as a consequence of Theorem \ref{thm: linearization} and Lemmas \ref{lem: same_model}, \ref{lm: conjisLandcont}, for two typical GIETs of class $C^3$ with zero mean non-linearity satisfying the EC Condition to be topologically conjugated, it is sufficient to assume that they have the same boundary and the same combinatorial rotation number.

As an intermediate step in the proof of our main result, we show in Theorem \ref{thm: linearization} that for a typical $C^3$ GIET with zero mean non-linearity, there exists an \emph{affine interval exchange transformation} (AIET), i.e., a piecewise affine GIET, which is $C^1$ conjugated to the original GIET. It is worth mentioning that this AIET is not uniquely determined. In this case, there will be uncountably many possible \emph{affine models}, i.e., AIETs in the smooth conjugacy class of the GIET. In fact, the logarithms of slopes of these affine models form an affine vector space whose dimension is equal to that of the stable space of the \emph{Zorich height cocycle} (see Subsection \ref{sc:notations} for precise definitions).

Let us point out that, by Theorems \ref{thm: affine_shadow}, \ref{thm: linearization} and a recent result by C. Ulcigrai and the second author \cite{trujillo_affine_2022}, a typical GIET $f$ with zero mean non-linearity is topologically conjugated to a unique IET, but it is generally not smoothly conjugated to it (unless $\omega_f = 0$, where $\omega_f$ is given by Theorem \ref{thm: affine_shadow}). 

The above results apply to the case of piecewise $C^3$ circle homeomorphisms, which, as mentioned before, can be naturally identified with $C^3$ GIETs of combinatorial data of rotation type (see Section \ref{sc: p-homeomorphisms}). In Theorem \ref{thm: circle_thm}, we prove that two typical break-equivalent circle homeomorphisms in this class are $C^1$ conjugated to each other. Let us stress that, as mentioned before, the notion of typical in this statement involves not only restrictions in the rotation number of the transformation (when seen as a circle map) but properties of the combinatorial rotation number of the map (when seen as a GIET). 

\subsection{Outline of the article}
The following section contains the preliminaries and notations used throughout this work. The notations adopted in this work for classical objects related to interval exchange transformations (e.g., length vector, intervals of continuity, objects associated with Rauzy-Veech induction, and Zorich acceleration) are exhibited in shortened form in Section \ref{sc:notations}. We refer to the lecture notes \cite{viana_ergodic_2006, yoccoz_echanges_2009, zorich_flat_2006}  for a detailed account of these objects and their properties.

Sections \ref{sc:GIETs} and \ref{sc: p-homeomorphisms} contain facts and additional background on GIETs and piecewise smooth circle homeomorphisms. Particularly, on the boundary operator and a canonical identification of circle diffeomorphisms with breaks and GIETs. 
In Section \ref{sc: defR}, we present an acceleration of the Zorich cocycle, which plays a crucial role in the proofs of the results of this article.

In Section \ref{sc: results}, we state our main results precisely, while in Section \ref{sc: outline}, we provide an outline and proofs (relying on some auxiliary propositions, which are proved in the subsequent sections) of these results. Sections \ref{sc: existenceofaffineshadow} and \ref{sc: conjugacy} contain the proofs of these auxiliary results.

\begin{center}
\textsc{Acknowledgements}
\end{center}
The authors would like to thank Corinna Ulcigrai for presenting the problem and for many fruitful discussions. Both authors were supported by the \emph{Swiss National Science Foundation} Grant $200021\_188617/1$. The first author was also supported by \emph{National Science Centre (Poland)} grant OPUS $2022/45/B/ST1/00179$. The second author was partially supported by the UZH Postdoc Grant, grant no. FK-23-133, during the preparation of this work.

\section{Preliminaries}
\label{sc:preliminaries}
\subsection{Notations}\label{sc:notations}

An interval exchange transformation (IET) is a bijection of an interval $I$ (often $I=[0,1)$), which is a right-continuous piecewise translation with a finite number of discontinuities. It is often parametrized by the length vector, which contains the lengths of the exchanged intervals, and a permutation describing the exchange order. The exchanged intervals are indexed by a finite alphabet. We will adopt the following notations.

\begin{center}
\begin{longtable}{|>{\sc\footnotesize}l |>{\hspace{0.5pc}}l|}\hline
	Finite alphabet  & $\mathcal A$ of $d\ge 2$ elements\\\hline
	Length vector & $\lambda\in\R_+^{\A}$ or $\lambda\in\Delta_{\mathcal A}:=\{\lambda\in\R_+^{\A} \mid  \ |\lambda|=\sum_{\alpha\in\A}\lambda_\alpha=1\}$ \\\hline
	Permutation & $\pi\in \mathfrak G_{\mathcal A}:=\{(\pi_0,\pi_1) \mid \pi_0,\pi_1:\mathcal A\to\{1,\dots,d\}\text{ bijections}\}$\\\hline
	Irreducible permutations & $\mathfrak G_{\mathcal A}^0:=\{\pi\in\mathfrak{G}_{\A}\mid \pi_1\circ\pi_0^{-1}(\{1,\ldots,k\})\!=\!\{1,\ldots,k\}\Rightarrow k=d\}$\\\hline
	Space of IETs & $\R_+^{\A}\times\mathfrak G_{\mathcal A}^0$\\\hline
	Space of normalized IETs & $\Delta_{\mathcal A}\times\mathfrak G_{\mathcal A}^0$ \\\hline
\end{longtable}
\end{center}
\vspace{-0.5cm}

Elements in $ \R_+^{\A}\times\mathfrak G_{\mathcal A}^0$ are in one-to-one 
correspondence with \emph{irreducible} IETs defined on intervals of the form 
$[0, |\lambda|)$, that is, with IETs not admitting invariant subintervals of 
the form $[0, t)$, with $0 < t < |\lambda|$. {We endow the space of 
(normalized) IETs with the product of Lebesgue and counting measures. Notice 
that every IET preserves the Lebesgue measure on its interval of definition. 
Moreover, almost every irreducible IET is uniquely ergodic (see 
\cite{masur_interval_1982, veech_gauss_1982}).}

Given $(\lambda,\pi)\in \R_+^{\A}\times\mathfrak G_{\mathcal A}^0$, which encodes an 
IET $T:[0,|\lambda|)\to [0,|\lambda|)$, we associate the following objects.

\begin{center}
\begin{longtable}{|>{\sc\footnotesize}l |>{\hspace{0.5pc}}m{0.59\textwidth}|}\hline
Intervals exchanged  & $\{I_\alpha(T)\}_{\alpha\in\mathcal A}$, where $|I_\alpha(T)|=\lambda_\alpha$\\\hline
Endpoints of the partition & $0=u_0(T)<u_1(T)<\ldots<u_{d-1}(T)<u_d(T)=|\lambda|$\\\hline
Translation matrix & $\Omega_\pi: \R^\A \to \R^\A,$ where \newline \phantom{} \hskip0.1cm $\Omega_{\alpha, \beta} = \left\{ \begin{array}{cl} +1 & \text{if } \pi_1(\alpha) > \pi_1(\beta) \text{ and } \pi_0(\alpha) < \pi_0(\beta), \\
-1 & \text{if } \pi_1(\alpha) < \pi_1(\beta) \text{ and } \pi_0(\alpha) > \pi_0(\beta), \\
0 & \text{otherwise.} 
\end{array}\right.$ \\\hline
\end{longtable}
\end{center}
\vspace{-0.5cm}

A classical induction procedure, known as \emph{Rauzy-Veech induction}, associates to a typical IET $T$ another IET with the same number of intervals by inducing $T$ in some appropriate subinterval. This induction can be repeated infinitely many times for IETs verifying the so-called \emph{Keane's condition} (see \cite{keane_interval_1975}). Moreover, this procedure naturally induces a directed graph structure on $\mathfrak{G}^0_d$ called the \emph{Rauzy graph}. Each connected component in this graph is called a \emph{Rauzy class}. The infinite path in the Rauzy graph associated with an IET $T$ is called \emph{combinatorial rotation number} of $T$.

\begin{center}
\begin{longtable}{|>{\sc\footnotesize}m{0.4\textwidth} |>{\hspace{0.5pc}}m{0.53\textwidth}|}\hline
IETs satisfying Keane's condition & $X_{\mathcal A}^+\subset\R_+^{\A}\times\mathfrak G_{\mathcal A}^0$ \\\hline
 Winner and loser symbols & $\alpha_\epsilon = \pi_\epsilon^{-1}(d)$, for $\epsilon \in \{0, 1\}$ \\\hline
Type of $(\lambda, \pi) \in X_{\mathcal A}^+$ & $\epsilon(\lambda, \pi) = \left\{ \begin{array}{cl}
0 & \text{ if } \lambda_{\alpha_0} > \lambda_{\alpha_1}, \\ 
1 &\text{ if }\lambda_{\alpha_0} < \lambda_{\alpha_1}.
 \end{array}\right.$\\\hline

 Rauzy-Veech induction & \hskip-0.2cm$\Function{\mathcal{RV}}{X_{\mathcal A}^+}{X_{\mathcal A}^+}{T}{T\mid_{\left[0, |\lambda| -  \lambda_{\alpha_{\epsilon(T)}} \right)}}$ \\\hline
 Combinatorial rotation number & $\gamma(T)$ \\\hline
Cocycle notation & For $F: X \to X$, $\phi: X \to GL(d, \Z)$ and $n\geq 0$, \newline \phantom{} \hskip0.15cm $\phi^{(n)}(x) = \phi(F^{n - 1}(x)) \dots  \phi(x).$  \\\hline 
Rauzy-Veech matrix & $\Function{A}{X_{\mathcal A}^+}{SL(d,\Z)}{T}{Id+ E_{\alpha_{\epsilon(T)}, \alpha_{1 - \epsilon(T)}}}$ \\\hline

\end{longtable}
\end{center}
\vspace{-0.5cm}

The combinatorial rotation number of an infinitely renormalizable IET is 
\emph{infinitely complete ($\infty$-complete)}, i.e., each symbol in the 
alphabet appears as the winner symbol infinitely many times. Moreover, {there 
exists at least one IET verifying Keane's condition associated with any 
$\infty$-complete path, and this correspondence is one-to-one when restricted 
to uniquely ergodic IETs (see Corollary 5 and Proposition 6 in 
\cite{yoccoz_echanges_2005}).   
}

The Rauzy-Veech matrices are useful in describing several aspects of this induction procedure. In fact, given $n \ge 0$, $T  = (\lambda, \pi) \in X_\A^+$ and denoting $\mathcal{T} = \mathcal{RV}^n(T)$, the cocycles $A^{-1}, {}^TA: X_{\mathcal A}^+\to SL(d,\Z)$ associated to $A$ verify the following. 
\begin{itemize}
\item $(|I_\alpha(\mathcal{T})|)_{\alpha \in \A} = (A^{-1})^{(n)}(T) \lambda.$
\item If $\left(\mathcal{T}\mid_{I_\alpha(\mathcal{T})}\right)_{\alpha \in \A}= \left(T^{h_\alpha}\mid_{I_\alpha(\mathcal{T})}\right)_{\alpha \in \A}$ then $(h_\alpha)_{\alpha \in \A} = \big({}^TA\big)^{(n)}(T) \overline{1},$ where $\overline{1} \in \R_+^\A$ is the vector whose coordinates are all equal to $1$. 
\end{itemize} 

Sometimes it will be helpful to consider normalized variants of this procedure. For this purpose, we introduce the following notation.

\begin{center}
\begin{longtable}{|>{\sc\footnotesize}m{\FirstColumn} |>{\hspace{0.5pc}}l  |}\hline
	Normalized vector & $\tilde\xi:=\frac{\xi}{|\xi|}$, where $\xi\in \R_+^{\mathcal A}$  \\\hline 
	Normalized operator& $\tilde D(\xi):=\widetilde{D(\xi)}$, where $D$ is a positive matrix\\\hline
	Normalized IETs satisfying Keane's condition &  $X_{\mathcal A}\subseteq \Delta_{\A}\times\mathfrak G_{\mathcal A}^0$\\\hline
	Normalized RV-induction & $\widetilde{\mathcal{RV}}:X_{\mathcal A}\to X_{\mathcal A}$\\\hline
\end{longtable}
\end{center}

The normalized RV-induction $\widetilde{\mathcal{RV}}$ admits an infinite (but no finite) invariant measure equivalent to the Lebesgue measure. However, an appropriate acceleration of this procedure, known as \emph{Zorich map} and that we denote by $\widetilde{\mathcal{Z}}$, possesses a unique ergodic invariant probability measure $\mu_{\tilde{\mathcal Z}}$ equivalent to the Lebesgue measure.

\begin{center}
\begin{longtable}{|>{\sc\footnotesize}l |>{\hspace{0.5pc}}l|}\hline
	Zorich acceleration time & $z:X_{\mathcal A}^+\to \N$\\\hline
	Zorich induction & \hskip-0.2cm$\Function{\mathcal Z}{X_{\mathcal A}^+}{X_{\mathcal A}^+}{(\lambda,\pi)}{\mathcal{RV}^{z(\lambda,\pi)}(\lambda,\pi)}$\\\hline
	Zorich renormalization & \hskip-0.2cm$\Function{\tilde{\mathcal Z}}{X_{\mathcal A}}{X_{\mathcal A}}{(\lambda,\pi)}{\widetilde{\mathcal{RV}}^{z(\lambda,\pi)}(\lambda,\pi)}$\\\hline
	Zorich matrix & \hskip-0.2cm$\Function{B}{X_{\mathcal A}^+}{SL(d,\Z)}{(\lambda,\pi)}{\prod_{i = 0}^{z(\lambda, \pi) - 1}A(\mathcal{RV}^i(\lambda,\pi))}$\\\hline
	Zorich length cocycle & $B^{-1}:X_{\mathcal A}^+\to SL(d,\Z)$\\\hline
	Zorich height cocycle & ${}^TB:X_{\mathcal A}^+\to SL(d,\Z)$.\\\hline
\end{longtable}
\end{center}

As before, the cocycles above describe several aspects of this accelerated induction. Moreover, they are integrable with respect to the measure $\mu_{\tilde{\mathcal Z}}$ and thus admit an Oseledets filtration. 

\begin{center}
	\begin{longtable}{|>{\sc\footnotesize}m{\FirstColumn}|>{\hspace{0.5pc}}l|}\hline
		Osdeledet's filtration for the height cocycle   & $E^s(\lambda,\pi)\subset E^{cs}(\lambda,\pi)\subset\R^{\A}$ \\\hline
		Osdeledet's filtration for the length cocycle  & $F^s(\lambda,\pi)\subset F^{cs}(\lambda,\pi)\subset \R^{\A}$ \\\hline
	\end{longtable}
\end{center}

However, having an Oseledets splitting associated with these cocycles will be helpful. For this purpose, we consider a natural extension of the Zorich renormalization.
\begin{center}
\begin{longtable}{|>{\sc\footnotesize} m{3.5cm} | >{\hspace{0.5pc}}l|}\hline
	Extended space of \newline parameters & $\mathfrak X_{\mathcal A}^{+}=\left\{(\tau,\lambda,\pi) \in \R_+^\A \times  X_{\mathcal A}^{+} \,\left|\, \begin{array}{l}  \, \sum_{\pi_0(\alpha)\le k}\tau_{\alpha}>0;  \\ \sum_{\pi_1(\alpha)\le k}\tau_{\alpha}<0; \\
	 \text{for }1\le k \le d - 1. \end{array}\right\}\right.$\\\hline
	Extended Zorich \newline induction & \hskip-0.2cm$\Function{{\mathcal Z}_{\textup{ext}}}{\mathfrak X_{\mathcal A}^{+}}{\mathfrak X_{\mathcal A}^{+}}{(\tau,\lambda,\pi)}{(B^{-1}(\lambda,\pi)\tau,\mathcal Z(\lambda,\pi))}$\\\hline
	Normalized extended space of parameters &  $\mathfrak X_{\mathcal A} = \left\{(\tau,\lambda,\pi)\in\mathfrak X_{\mathcal A}^{+}\mid (\lambda,\pi)\in X_{\mathcal A}; \, \langle \Omega_{\pi}\tau,\lambda \rangle=1\right\} $\\\hline
	Extended Zorich \newline renormalization &  \hskip-0.2cm$\Function{\tilde{\mathcal Z}_{\textup{ext}}}{\mathfrak X_{\mathcal A}}{\mathfrak X_{\mathcal A}}{(\tau,\lambda,\pi)}{(|B^{-1}(\lambda,\pi)\lambda |B^{-1}(\lambda,\pi)\tau,\tilde{\mathcal Z}(\lambda,\pi))}$ \\\hline
\end{longtable}
\end{center}

The extended Zorich renormalization $\tilde{\mathcal Z}_{\textup{ext}}$ also admits an invariant probability measure $\mu_{\tilde{\mathcal Z}_{\textup{ext}}}$ and projects to $\tilde {\mathcal{Z}}$ via a natural projection $p:\mathfrak X_{\mathcal A}\to X_{\mathcal A}$. 
	The length and height cocycles can be thus naturally extended to $\mathfrak 
	X_{\mathcal A}$. For the sake of simplicity, we denote them by the same 
	letters. 
	Since $\tilde{\mathcal Z}_{\textup{ext}}$ is invertible and the cocycles are integrable, for a.e. $(\tau,\lambda,\pi)\in \mathfrak X_{\mathcal A}$ we obtain the following. 

\begin{center}
\begin{longtable}{|>{\sc\footnotesize}m{6cm} |>{\hspace{0.5pc}}l|}\hline
Osdeledet's splitting for the height cocycle  ${}^T B: \mathfrak X_{\mathcal A} \to SL(d, \Z)$ & $E^s(\tau,\lambda,\pi)\oplus E^c(\tau,\lambda,\pi)\oplus E^u(\tau,\lambda,\pi)=\R^{\A}$ \\\hline
Osdeledet's splitting for the length cocycle  $B^{-1}: \mathfrak X_{\mathcal A} \to SL(d, \Z)$ & $F^s(\tau,\lambda,\pi)\oplus F^c(\tau,\lambda,\pi)\oplus F^u(\tau,\lambda,\pi)=\R^{\A}$ \\\hline
\end{longtable}
\end{center}

Let us point out that, for a.e. $(\tau,\lambda,\pi)\in\mathfrak X_{\mathcal A}$, we have $E^s(\tau,\lambda,\pi)=E^s(\lambda,\pi)$ and $E^{c}(\tau,\lambda,\pi)\oplus E^s(\tau,\lambda,\pi)=E^{cs}(\lambda,\pi)$.


The operators defined above, as well as Keane's condition, extend naturally to 
\emph{generalized interval exchange transformations} (GIETs), that is, to 
piecewise smooth right continuous bijections of an interval with a finite 
number of discontinuities. We also consider \emph{affine interval exchange 
transformations} (AIETs), which are piecewise linear GIETs. {In the following, 
we restrict our discussion to infinitely renormalizable GIETs.}

For a given GIET, we denote the associated partition, the endpoints of the 
intervals in this partition, its combinatorial rotation number, and the 
associated inductions/renormalizations, as in the case of IETs. A GIET $f$ is 
parametrized by the permutation $\pi$ describing the order in which the 
intervals are exchanged, the lengths of the exchanged intervals 
$\{I_\alpha(f)\}_{\alpha \in \A}$ before and after applying the function, and a 
rescaling of the different branches $f\mid_{I_\alpha}: I_\alpha \to 
f(I_\alpha)$. {As before, we will consider only maps whose underlying 
permutation is irreducible.}

\begin{center}
\begin{longtable}{|>{\sc\footnotesize}m{4.8cm} |>{\hspace{0.5pc}}l|}\hline
Zoom operator & $\Xi: \textup{Diff}^r([a, b], [c, d]) \to \textup{Diff}^r([0, 1])$  \\\hline
Space of GIETs & $\mathcal{X}^r_\A = \mathfrak R \times \mathfrak{G}^0_\A \times \textup{Diff}^r([0, 1])^\A,  $ \\& where $\mathfrak{R}=\{(u,w)\in\R_+^{\A}\times \R_+^{\A}\mid |u|=|w|\}$\\\hline
Average log-slope of a GIET & $L(f) = (L_\alpha(f))_{\alpha \in \A}$, \quad where $f \in \mathcal{X}^r_\A$ and \\\hline
& $L_\alpha(f) = \ln  \left( \frac{1}{|I_\alpha(f)|} \int_{I_\alpha(f)} Df\mid_{I_\alpha(f)}(s)ds \right)$ \\ \hline
Space of AIETs & $\textup{Aff}_\A = \mathfrak R \times \mathfrak{G}^0_\A $ \\\hline
Log-slope vector of an AIET& $L(S) \in\R^{\mathcal A}$, \quad where $S\in \textup{Aff}_{\mathcal A}$\\\hline
AIETs with log-slope $\omega$ semi-conjugated to an IET $T$ & $\textup{Aff}(T,\omega):=  \left\{ S\in \textup{Aff}_{\A} \,\left|\, \begin{array}{l}  \, \gamma(T)=\gamma(S) \text{ and } \\ 
S \text{ has log-slope }  \omega \end{array}\right\}\right. $ \\\hline
\end{longtable}
\end{center}
 \vspace{-0.5cm}
 
  {
Combinatorial rotation numbers of GIETs with $d$ intervals belong to the space of all admissible infinite paths (with respect to the directed graph structure) in the Rauzy graph $\mathfrak{G}^0_d$, which we denote by $\mathfrak{C}^0_d$. Notice that any admissible finite path $\gamma$ on $\mathfrak{G}^0_d$ determines a positive measure set of IETs whose combinatorial rotation number starts by $\gamma$. Using this fact, we can naturally define a measure on $\mathfrak{C}^0_d$ so that, for any admissible finite path $\gamma$ on $\mathfrak{G}^0_d$, the measure of the set of admissible infinite paths starting by $\gamma$ is equal to the positive measure set of IETs associated with $\gamma$. This measure coincides with the pushforward of the Lebesgue measure on the set of IETs by the map $T \mapsto \gamma(T)$ and, in particular, gives full weight to the set of $\infty$-complete paths. 
 
Using this measure, we can define a notion of \emph{typicality for GIETs} as follows. We say that a property $P$ holds for typical maps in a given subset of GIETs on $d$ intervals if there exists a full-measure set of combinatorial rotation numbers $\mathfrak{C} \subseteq \mathfrak{C}^0_d$ such that any map in this class whose combinatorial rotation number belongs to $\mathfrak{C}$ satisfies the property $P$. As an abuse of language, we sometimes state this by saying that a typical map in the given class verifies the property $P$.
 
We say that a GIET is \emph{irrational} if it is infinitely renormalizable and its rotation number is $\infty$-complete.  By \cite[Proposition 7]{yoccoz_echanges_2005}, any irrational GIET is semi-conjugated to an IET with the same rotation number. It follows trivially from the previous remarks that a typical GIET is irrational and, therefore, typical GIETs are semi-conjugated to IETs.

 Let us point out that, in contrast to the case of IETs, not all infinitely renormalizable GIETs have an $\infty$-complete rotation number and, in parameter space, is expected for a generic GIET not to have $\infty$-complete combinatorial rotation number. On the other hand, infinitely renormalizable GIETs without wandering intervals and irreducible underlying permutations are expected to be irrational. 
} 

Let us recall that given two GIETs $f,g: I \to I$, we say that $f$ is \emph{semi-conjugated} to $g$  if there exists a non-decreasing continuous surjective map $h: I \to I$ such that $f\circ h=h\circ g$ and $h(I_\alpha(f)) = I_\alpha(g)$, for all $\alpha \in \A$.

It follows from our definition of GIET that for any $f\in\mathcal X^r_{\mathcal A}$, there exists $c>0$ such that $D_-f, D_+f>c$, wherever those derivatives are defined. In other words, the GIETs we consider do not possess critical points.

Notice that although the height and length cocycles also extend trivially to the GIET setting (in fact, they depend only on the combinatorial rotation number of a GIET), and the height cocycle still describes the return times for the induced transformation, the length cocycle no longer describes the lengths of the intervals in the partition of the induced transformation. 

\subsection{Generalized interval exchange transformations}
\label{sc:GIETs}

Given a GIET $f: I \to I$ such that $D(\log(Df))\in L^1(I)$ we denote its \emph{mean-non-linearity} by
\[ \mathcal{N}(f) = \int_{I} D\log(Df(x))\,dx.\]
Given a partition $\{I_\alpha\}_{\alpha \in \A}$ of an interval $I$, we denote by 
$C^r(\bigsqcup_{\alpha \in \A} I_\alpha)$ the set of real functions on $I$ such that, for each $\alpha \in \A$, the restriction to $I_\alpha$ extends to a $C^r$- function on $\overline{I_\alpha}$.  

Recall that for any IET $T = (\lambda, \pi) \in \Delta_{\A} \times \mathfrak{G}^0_d$  the endpoints of the associated partition  are identified with the singular points $\{s_1, \dots, s_\kappa\}$ in the associated translation surface (see, e.g.,  \cite{yoccoz_echanges_2009, viana_ergodic_2006}), where 
\[\kappa = \dim(\textup{Ker}(\Omega_\pi)) + 1.\]
 The relations between the partition's endpoints and the translation surface's singularities can be encoded as follows.  Let $p = \pi_1 \circ \pi_0^{-1}$ denote the \emph{monodromy invariant} of $\pi$ and define $\sigma: \{0, 1, \dots, d\} \to \{0, 1, \dots, d\}$ by 
\begin{equation}
\label{eq: perm_sigma}
\sigma(j) = \left\{  
    \begin{array}{lcl} 
        p^{-1}(1) - 1 & \text{ if } & j = 0, \\
       d& \text{ if } & p(j) = d, \\
        p^{-1}(p(j) + 1) - 1 & \text{  } & \text{otherwise}. \\
     \end{array}  \right.
\end{equation}
Then $\sigma$ is a well-defined bijection, having exactly $\kappa$ different orbits, and each orbit corresponds to endpoints of the partition that are identified with the same singularity in the associated translation surface. Moreover, $\sigma$ allows to define a base for $\textup{Ker}(\Omega_\pi)$ given by the vectors 
\[ \{ \lambda(\mathcal{O}) \in \R^\A \mid \mathcal{O} \text{ is an orbit of } \sigma \text{ not containing } 0 \}, \]
where 
\[ \lambda(\mathcal{O})_{\pi_0(j)} = \left\{
    \begin{array}{lcl}
        1 & \text{ if } & j \in \mathcal{O} \text{ and } j - 1 \notin \mathcal{O}, \\
        -1 & \text{ if } & j \notin \mathcal{O} \text{ and } j - 1 \in \mathcal{O}, \\
        0 & \text{  } & \text{otherwise}. \\
    \end{array} 
\right.\]
For details, we refer the interested reader to \cite{viana_ergodic_2006}.

 Following \cite{marmi_linearization_2012}, given an irreducible permutation 
 $\pi \in \mathfrak{G}^0_d$, a partition $\{I_\alpha\}_{\alpha \in \A}$ of $I$, 
 and denoting by $s(u_i)$ the singularity associated to the endpoints $u_i$, 
 for $i = 0, \dots, d$ of the underlying IET, we define the \emph{boundary 
 operator} associated to $\pi$ and $\{I_\alpha\}_{\alpha \in \A}$ as 
\[\Function{\mathfrak{B}}{C^0\Big(\bigsqcup_{\alpha \in \A} I_\alpha\Big)}{\R^\kappa}{\varphi}{(\mathfrak{B}_s(\varphi))_{1 \leq s \leq \kappa}}, \qquad \mathfrak{B}_s(\varphi) = \sum_{\substack{ s(u_i) = s, \\ 0 \leq i \leq d}} \partial \varphi(u_i), \]
where $\partial \varphi(u_i)$ denotes the difference between the right and left limits of $\varphi$ at $u_i$ (assuming $\varphi$ is defined as $0$ outside of $I$), namely
\[ \partial \varphi(u_i) = \left\{ 
    \begin{array}{lcl}
        \varphi(u_0) & \text{ if } & i = 0, \\
        \lim_{x \to u_i^+} \varphi(x) - \lim_{x \to u_i^-} \varphi(x)   & \text{ if } & 0 < i < d, \\
        -\lim_{x \to u_d^-} \varphi(x) & \text{ if } & i = d.\\
    \end{array} 
\right.\]
We define the \emph{boundary} $\mathcal{B}(f)$ of a GIET $f$  over an irreducible permutation $\pi \in \mathfrak{G}^0_d$,  as 
\begin{equation}
\label{eq: boundary}
\mathcal{B}(f) = \mathfrak{B}(\log Df),
\end{equation}
where $\mathfrak{B}$ is the boundary operator associated to $\pi$ and $\{I_\alpha(f)\}_{\alpha \in \A}$. We say that two GIETs $f$ and $g$ with the same combinatorial data are \emph{boundary-equivalent} iff $\mathcal{B}(f)=\mathcal{B}(g)$.

 For AIETs, the boundary operator is closely related to the projection of the log-slope vector to the space $\Kpi$, where $\pi$ is the subjacent permutation.  Indeed, if $S$ is an AIET with log-slope $\omega$ over $\pi \in \mathfrak{G}^0_d$ and $\mathcal{O}$ is the orbit of $\sigma_\pi$ associated to the singularity $s \in \{s_1, \dots, s_\kappa\}$, then 
\[ \mathfrak{B}_s(S) = \langle \omega, \lambda(\mathcal{O}) \rangle.\]
In particular, the log-slope vectors of two AIETs over $\pi \in \mathfrak{G}^0_d$ having the same boundary project to the same vector in $\Kpi$. 

The boundary operator for GIETs verifies the following.

\begin{proposition}
\label{prop: boundary_properties}
Let $f $  be a $C^1$ GIET  with genus $\kappa$ satisfying Keane's condition. Then the following holds.
\begin{enumerate}
\item $\mathcal{B}(\mathcal{RV}(f)) = \mathcal{B}(f)$.
\item $\mathcal{N}(f) = 0$  if and only if $\sum_{1 \leq s \leq \kappa} b_s = 0,$ where  $(b_s)_{1 \leq s \leq \kappa} = \mathcal{B}(f).$
\item Let $g$ be a $C^1$ GIET over $\pi$. If $f$ and $g$ are $C^1$ conjugated, then $\mathcal{B}(f) = \mathcal{B}(g)$.
\end{enumerate}
\end{proposition}

For a proof of the properties above, see \cite{ghazouani_priori_2023}. By the first assertion in the above proposition and the previous discussion on the boundary operator for AIETs, we have the following.

\begin{proposition}
\label{prop: boundary_AIETs}
Given two AIETs $f$, $g$ over $\pi \in \mathfrak{G}^0_d$, with zero mean non-linearity and log-slope vectors $\omega_f,$ $\omega_g$, 
\[ \mathcal{B}(f) = \mathcal{B}(g) \iff  \pi_{\Kpi}(\omega_f) =  \pi_{\Kpi}(\omega_g).\]
\end{proposition}

We now define an equivalent of a length cocycle for GIETs.
Let $f$ be a GIET and $T$ be such that $\gamma(f)=\gamma(T)$. Let $A^n:=A(\lambda,\pi)\cdot\ldots\cdot A(\mathcal{RV}^{n-1}(\lambda,\pi))$. For any $n\in\N$ consider a $d\times d$ matrix  $A^n(f)$  defined as follows
\begin{equation}\label{matrixf}
A^n_{\alpha\beta}(f):=\sum_{i=1}^{A_{\alpha\beta}^n}\frac{|f^{m_i(\alpha, \beta)}(I_\alpha(\mathcal{RV}^nf))|}{|I_\alpha(\mathcal{RV}^nf)|},
\end{equation}
where $m_i(\alpha, \beta)$ is the $i$-th return time of the interval $I_\alpha(\mathcal{RV}^nf)$ to the interval $I_\beta(f)$ via $f$.
Note that, by the mean value theorem, we have
\begin{equation}\label{matrixS}
A_{\alpha\beta}^n(f):=\sum_{i=1}^{A_{\alpha\beta}^n}Df^{m_i(\alpha, \beta)}(x_{i}^{n}(\alpha,\beta)),
\end{equation}
for some $x^{n}_{i}(\alpha,\beta)\in I_\alpha(\mathcal RV^nf)$.

The importance of the matrices defined above follows from the fact that
\begin{equation}\label{lengthrelation}
\big(|I_{\alpha}(f)|\big)_{\alpha\in\mathcal A}=A^n(f)\big(|I_{\alpha}(\mathcal{RV}^nf)|\big)_{\alpha\in\mathcal A}.
\end{equation}

\subsection{Piecewise smooth circle homeomorphisms}
\label{sc: p-homeomorphisms}

Our results apply to the setting of piecewise smooth circle homeomorphisms with a finite number of branches, but as their formulations and hypotheses will be slightly different, we introduce the following notations. 

A \emph{piecewise smooth circle homeomorphism} $f: \T \to \T$ is a smooth orientation preserving homeomorphism, differentiable away from countable many points, so-called \emph{break-points}, at which left and right derivatives, denoted by $Df_-$, $Df_+$ respectively, exist but do not coincide. For any $x \in \T$, we define the \emph{jump ratio} of $f$ at $x$ by 
\[ \sigma_f(x) = \frac{Df_-(x)}{Df_+(x)}.\]
We denote the \emph{set of break points} of $f$ by 
\[BP(f) = \{x \in \T \mid \sigma_f(x) \neq 1 \}.\]
Notice that if the map has a finite number of break points, then the product of  all jump ratios at the break points is equal to one, namely,
\begin{equation}
\label{eq: product_jump_ratios}
\prod_{x \in BP(f)} \sigma_f(x) = 1.
\end{equation}
Two piecewise smooth circle homeomorphisms $f$ and $g$ are said to be \emph{break-equivalent} if they are topologically conjugated by a map $h \in \textup{Hom}(\T)$, verifying $h \circ f = g \circ h$, such that $h(BP(f)) = BP(g)$ and $\sigma_f(x) = \sigma_g(h(x)),$ for any $x \in BP(f)$.

We denote by $P_d^r(\T)$ the set of $C^r$ piecewise smooth circle homeomorphisms with exactly $d$ breaks, having a break at $\varphi(0)$, and such that the orbits of the breaks are two by two disjoint. 

For any $d \geq 3$, we identify maps in $P_{d - 1}^r(\T)$ with GIETs on $d$ intervals as follows.  Let 
\begin{equation}
\label{eq: parametrization}
\Function{\varphi}{[0, 1)}{\T}{x}{e^{2\pi i x}}.
\end{equation}
Then, for any $f \in P_{d - 1}^r$, the map $F = \varphi^{-1} \circ f \circ \varphi$ is an infinitely renormalizable GIET of class $C^r$ on $d$ intervals. Notice that this map cannot be seen as a GIET on a smaller number of intervals since, by assumption, $f$ has exactly $d$ break points, and the orbits of these breaks are two by two disjoint. 

In this case, the boundary $\mathcal{B}(F)$ defined in \eqref{eq: boundary} is nothing more than the vector given by the logarithm of the jump ratios of $f$ at the break points. Indeed, by definition of $F$, it is easy to see that the combinatorial datum $\pi = (\pi_0, \pi_1)$ of $F$ satisfies
 \begin{equation}
\label{eq: rotation_type}
\pi_1 \circ \pi_0^{-1}(i) - 1 = i + k \textup{ (mod } d), \quad \text{ for } i = 1, \dots, d, 
\end{equation}
for some $k \in \{0, \dots, d - 2\}.$ {We refer to GIETs with such 
combinatorial data as \emph{GIETs of rotation type} and to any $\pi$ like in 
\eqref{eq: rotation_type} as the \emph{combinatorial data of rotation type}.} 
In this case, an explicit computation shows that the permutation $\sigma$ 
defined in \eqref{eq: perm_sigma} is given by 
\begin{equation}
\label{eq: perm_sigma_rotation}
\sigma(j) = \left\{  
    \begin{array}{lcl} 
        d- k - 1 & \text{ if } & j = 0, \\
       d & \text{ if } & j = d - k - 1, \\
       0 & \text{ if } & j = d, \\
       j & \text{  } & \text{otherwise}. \\
     \end{array}  \right.
\end{equation}
Denote by $0 = u_0 < u_1 < \cdots < u_d = 1$ the endpoints of the partition associated to $F$. Notice that $BP(f) = \{ \varphi(u_j) \mid 0 \leq j < d;\, j \neq d - k -1 \}$.

By \eqref{eq: perm_sigma_rotation}, $u_0, u_{d - k - 1}$ and $u_d$ correspond to the same singularity (in the associated translation surface), which we denote by $s_1$. In addition, since $\varphi(u_{d - k - 1}) \notin BP(f)$, 
\begin{align*}
\mathfrak B_{s_1}& (\log DF)  =   \log D_+F(u_0) +  \log D_+F(u_{d - k - 1})  -  \log D_-F(u_{d - k - 1})  -  \log D_-F(u_1) \\
& = \log D_+f(\varphi(0)) + \log D_+f(\varphi(u_{d - k - 1})) - \log D_-f(\varphi(u_{d - k - 1}))  - \log D_-f(\varphi(0)) \\
& = \log \sigma_f(\varphi(0)).
\end{align*} 
Furthermore, the remaining $d - 2$ singularities, which we denote by $s_2, \dots, s_{d - 1}$, are in one-to-one correspondence with the remaining endpoints. By associating $s_j$ to $u_{j - 1}$ if $2 \leq j \leq d - k - 1$, and to $u_j$ if $d - k - 1 < j < d$, we can express the boundary of $f$ as
\begin{equation}
\label{eq: boundary_p_homeos}
\mathcal{B}(F)_j = \left\{  
    \begin{array}{lcl} 
        \log \sigma_f(\varphi(u_{j - 1})) & \text{ if } & j \leq d - k - 1, \\
        \log \sigma_f(\varphi(u_j)) & \text{ if } & j > d - k - 1.\\
     \end{array}  \right.
\end{equation}
The argument above shows a slightly more general fact, namely, if $T$ is a GIET on $d$ intervals with combinatorial datum verifying \eqref{eq: rotation_type} and denoting by $0 = u_0 < u_1 < \cdots < u_{d + 1} = 1$, then we can express the boundary of $T$ as
\begin{equation}
\label{eq: boundary_rotation_type}
{\small
\mathcal{B}(T)_j = \left\{  
    \begin{array}{lcl} 
        \partial \log DT(u_{0}) + \partial \log DT(u_{d - k - 1}) + \partial \log DT(u_{1})  & \text{ if } & j = 1,\\
        \partial \log DT(u_{j - 1})  & \text{ if } & 1 < j \leq d - k - 1, \\
        \partial \log DT(u_{j})  & \text{ if } & d - k - 1 < j < d.\\
     \end{array}  \right.}
\end{equation}
Finally, let us mention that given a GIET $T$ on $d \geq 3$ intervals with a combinatorial data of rotation type, the circle homeomorphism induced by $T$ and given by $\varphi \circ T \circ \varphi^{-1}$ is a piecewise smooth circle homeomorphism with \emph{at most} $d$ breaks and thus it is not necessarily a map in $P_{d - 1}(\T)$. In fact, if the combinatorial datum of $T$ verifies \eqref{eq: rotation_type}, for some $k \in \{0, \dots, d - 2\}$, and $0 = u_0 < u_1 < \cdots < u_d = 1$ are the endpoints of the  associated partition of $[0, 1)$, then the circle homeomorphism $\varphi \circ T \circ \varphi^{-1}$ belongs to $P_{d - 1}(\T)$ if and only if $\partial \log DT (u_{d - k - 1}) = 0$ and $\mathcal{B}(T)_j \neq 0,$ for all $j = 1, \dots, d - 1$. 

Moreover, it follows from \eqref{eq: boundary_rotation_type} that for any two boundary-equivalent GIETs with combinatorial data of rotation type, the associated piecewise smooth homeomorphisms (having at most $d$ break points) have the same jump ratio in at least $d - 2$ break points. However, notice that it is still possible for two GIETs with combinatorial data of rotation type to be topologically conjugated and boundary-equivalent, while the associated piecewise smooth homeomorphisms are topologically conjugated but not break-equivalent (in this case, the associated circle maps would necessarily have $d$ breaks).


\subsection{An appropriate subsequence of Zorich times}\label{sc: defR}

We will consider an acceleration $\mathcal{R}: X \to X$ of the Zorich extension $\mathcal{Z}_{\textup{ext}}$, defined over a $\mu_\cZE$-full-measure subset $X \subset \DomE$, and of the form
\[ \mathcal{R}(\tau, \lambda, \pi) = \cZE^{\!\!\!\!\!\!\!\!N(\tau,\lambda, \pi)}(\tau, \lambda, \pi), \]
where $N: X \to \N$ is a measurable function. We denote by 
\[Q: X \to SL(d, \Z),\] the associated accelerated Zorich cocycle, which is given by 
\[Q(\tau, \lambda, \pi) = \prod_{i=0}^{N(\tau,\lambda, \pi)-1}B(\cZE^{\!\!\!\!\!\!\!i} \ \ (\tau, \lambda, \pi)),\]
and, as before, we denote for any $ n > m$,
\[ \HCT{m}{n}(\tau, \lambda, \pi) = \prod_{i=m}^{n-1} Q\left(\mathcal{R}^{i}(\tau, \lambda, \pi)\right),\]
\[ \HCT{-n}{-m}(\tau, \lambda, \pi)  = \big(\HCT{0}{n - m}\left(\mathcal{R}^{-n}(\tau, \lambda, \pi) \right)\big)^{-1}.\]
Note that $Q$ defines accelerated length and height cocycles, given by $Q^{-1}$ and ${}^T Q$, respectively. As in the case of $\mathcal Z_{\textup{ext}}$, we denote by $\tilde{\mathcal R}$ the normalized variant of $\mathcal R$. 

Given  $(\tau, \lambda, \pi)  \in X,$ we denote its \emph{iterates} with respect to $\mathcal{R}$ by $(\tau^n, \lambda^n, \pi^n)$, the \emph{induction intervals} by $(I^n_\alpha(\tau, \lambda, \pi))_{\alpha \in \A}$ and the associated \emph{heights} by $(q^n_\alpha(\tau, \lambda, \pi))_{\alpha \in \A}$.  If there is no risk of confusion, we will sometimes omit the base point $(\tau, \lambda, \pi)$ in all of the previous notations. 

The function $N$ will be picked so that $\mathcal{R}$ will satisfy the following.
\begin{proposition}
\label{prop: generic_condition}
For any $(\tau, \lambda, \pi)  \in X$ the following holds: 
\begin{enumerate}
\item \label{sublinear_times} The sequence $m_n = \sum_{k = 0}^{n -1} N(\tau^k, \lambda^k, \pi^k)$ is sublinear, that is, 
\[ \sup_{n \geq 1} \frac{m_n}{n} < +\infty.\]

\item \label{balanced_base} For any $n \geq 0$, $$\min_{\alpha, \beta \in \A} \frac{\lambda^n_\alpha}{\lambda^n_\beta} > c_0,$$ where $c_0 > 0$ is a constant independent of  $n$. 

More precisely, $\lambda^n\in A^{\gamma}(\R^{\mathcal A}_+)$, where $\gamma$ is a finite path in the Rauzy-Veech algorithm, independent of $n$ and $(\tau, \lambda, \pi)$, such that the associated matrix $A^{\gamma}$ is positive.

\item \label{bounded_central}  For any $n\geq 1,$ \[ C^{-1} \leq \left\|\HC{0}{n}\mid_{E^c(\tau, \pi, \lambda)} \right\| < C,\]
where $C > 0$ is a constant independent of $n$. 
\item \label{exp_growth} There exists $C>0$ such that $\sup_{n\in\N}\frac{\log\|\HCT{0}{n}\|}{n}<C$,
\item \label{Roth_type} For any $\tau > 0$, there exists $C > 0$ such that
\[ \left\| \HCT{n - 1}{n} \right\| \leq C \left\|  \HCT{0}{n} \right\|^\tau,\]
for every $n \geq 1$. In particular, by condition \ref{exp_growth}, there exists $C'>0$ such that
\[
\left\|\HCT{0}{n} \right\|\le C'e^{n\tau}.
\]
\end{enumerate}
\end{proposition}
\begin{proof}
	Conditions \ref{balanced_base} and \ref{bounded_central} are given by recurrence to an appropriate set of positive $\mu_{\mathcal Z}$ measure in $\Dom$ (for Condition \ref{balanced_base} we refer the reader to \cite{viana_ergodic_2006}, while for Condition \ref{bounded_central} to \cite{trujillo_affine_2022}). Condition \ref{sublinear_times} follows from the ergodicity of Zorich cocycle ($N$ is given by the visits to the set specified above), Conditions \ref{exp_growth} and \ref{Roth_type} follow from its integrability (which is inherited by the induced cocycle). We refer the reader to \cite{viana_lectures_2014} for a proof of this fact.
\end{proof}

Notice that, in view of Condition \ref{sublinear_times}, the Oseledet's splitting of the cocycle ${}^TB$ and ${}^T Q$ coincide.
 

In \cite{cunha_renormalization_2013}, the authors define a metric $d_{C^r}$ on $C^r$ piecewise smooth diffeomorphisms with $d$ branches and fixed combinatorial data (this metric is closely related to the one introduced in \cite{ghazouani_priori_2023}), where the distance $d_{C^r}(f, g)$, of any two maps $f$ and $g$ in this class, is given by
\[
\max_{\alpha\in\mathcal A}\left\{\left\|\Xi \big( f\mid_{I_\alpha(f)}\big)-\Xi \big( g\mid_{I_\alpha(g)} \big)\right\|_{C^r}+ \big||I_\alpha(f)|-|I_\alpha(g)|\big|+\big||f(I_\alpha(f))| -|g(I_\alpha(g)) |\big|\right\},
\]
where $\Xi$ is the {zoom operator}. This operator, defined in both articles mentioned above, associates to any homeomorphism between two bounded intervals its rescaling by affine transformations to a homeomorphism of the unit interval. More precisely, if $h: I \to J$ is a homeomorphism between two closed bounded intervals, then
\[ 
\Xi(h) = A_1 \circ h \circ A_2,
\]
where $A_1: J \to [0, 1]$, $A_2: [0, 1] \to I$ are bijective orientation preserving homeomorphisms.

%



The following is a direct consequence of \cite[Corollary 3.6]{marmi_linearization_2012}.

\begin{proposition}[Cohomological equation]
\label{prop: cohomological_equation0}
Let $f $  be a $C^0$ GIET satisfying Keane's condition. 
	Then, for any $\varphi \in C^0\big(\bigsqcup_{\alpha \in \A} I_\alpha(f)\big)$ verifying
\[ \sup_{n \geq 0} \big\| \mathcal S^{f}_n \varphi \big\| < +\infty,\]
where $\mathcal S^{f}_n\varphi $ denotes $n$-th Birkhoff sum of $\varphi$ w.r.t. $f$, there exists $\psi \in C^0([0, 1])$ such that
\[ \varphi = \psi \circ f - \psi.\]
\end{proposition}

\subsection{Exponential convergence of renormalizations}

The following condition, which concerns the exponential convergence (EC) of renormalizations of a GIET to the space of AIETs, summarizes the main assumption in our results. 

\begin{definition}[EC Condition]
\label{def:EC}
Let $f$ be a GIET semi-conjugated to a minimal IET $T = (\lambda, \pi)$. We say that $f$ satisfies the \emph{EC Condition} if there exists $\tau \in \R^\A$ such that $(\tau, \lambda, \pi)$ is Oseledet's generic w.r.t. to the Zorich cocycle, verifying
\begin{equation}\tag{EC}\label{EC} 
\begin{aligned}
  \max _{\alpha\in\mathcal A}\left\|\Xi(\cZ^n f)\mid_{I_\alpha(\mathcal Z^nf)}- Id\right\|_{C^1[0,1]} = O(c^n), \\
  |L^u(\cZ^nf)|,  |L^s(\mathcal Z^nf)| = O(c^n),
 \end{aligned}
 \end{equation}
for some $0 < c < 1$, where $L(\cZ^nf) =L^s(\cZ^nf) + L^c(\cZ^nf) + L^u(\cZ^nf)$ is decomposed with respect to the Oseledet's splitting of the height cocycle at $\cZ_{\textup{ext}}^n(\tau, \lambda, \pi)$.
\end{definition}

As mentioned in the introduction, it follows from a recent work by S. Ghazouani 
and C. Ulcigrai \cite{ghazouani_priori_2023} that this property holds for 
{typical $C^3$ GIETs of rotation type and typical minimal $C^3$ GIETs of 4 or 5 
intervals.} The authors in \cite{ghazouani_priori_2023} conjectured that the a 
priori bounds should hold for typical minimal $C^3$ GIET with $\mathcal N(f)=0$ 
without any assumption on the underlying combinatorics.  Moreover, they proved 
that this conjecture is true if a generalization (to AIETs of any combinatorial 
type) of a result by S. Marmi-P. Moussa and J.C. Yoccoz 
\cite{marmi_linearization_2012} concerning Birkhoff sums for AIETs holds. 

Subsequent renormalizations of a map satisfying the EC condition verify the following.
\begin{corollary}
	\label{cor: exp_convergence}
	For a.e. IET $T = (\lambda, \pi) \in \Delta_{\A} \times \mathfrak{G}^0_d$, for any GIET $f$ with $\gamma(f) = \gamma(T)$  and $\mathcal N(f)=0$ verifying the EC Condition, it holds that
	\[\max_{\alpha\in\mathcal A}\max_{x, y \in I^n_\alpha(f)} \frac{Df^{q^n_\alpha}(x)}{Df^{q^n_\alpha}(y)} = 1 + O(c^n),\]
	for some $0 < c < 1$.
\end{corollary}

\section{Statements of the results}\label{sc: results}

The following theorem is the main result of this article.
\begin{theorem}[Rigidity]
\label{thm: rigidity} 
For a.e. $T = (\lambda, \pi) \in \Delta_{\A} \times \mathfrak{G}^0_d$, any two boundary-equivalent GIETs $f$ and $g$ of class $C^3$, verifying 
\[\mathcal{N}(f) = 0 = \mathcal{N}(g), \qquad \gamma(f) = \gamma(T) = \gamma(g),\]
and satisfying the EC Condition, are $C^1$ conjugated.
\end{theorem}

We will prove the theorem above by showing that any two maps, $f$ and $g$, as in the statement above, are $C^1$ conjugated to the same AIET.  

For any IET $T$ and any $\omega\in \R^d$, we denote by $\textup{Aff}(T,\omega)$ the set of AIETs whose log-slope vector is $\omega$ and whose combinatorial rotation number is $\gamma(T)$.

To construct an affine model of a given mean non-linearity zero GIET $f$, we look for an AIET $S$ that shadows the orbit of the $f$ with respect to an appropriate acceleration of the Rauzy-Veech induction. Let us point out that such an affine model is not necessarily unique. 

\begin{theorem}[Affine shadow]
\label{thm: affine_shadow}
 For a.e. $T = (\lambda, \pi) \in \Delta_{\A} \times \mathfrak{G}^0_d$ and for any GIET $f$ of class $C^3$ with $\mathcal{N}(f)=0$ and $\gamma(f) = \gamma(T)$ verifying the EC Condition, there exists $\omega_f \in \Kpi$ such that, for any $S \in \textup{Aff}(T, \omega)$ with $\omega \in E^{cs}(\lambda, \pi)$ satisfying $\pi_{\Kpi}(\omega) = \omega_f,$  we have $\mathcal{B}(S) = \mathcal{B}(f)$ and
\[ \textup{d}_{C^1}(\tilde{\mathcal{R}}^nf, \tilde{\mathcal{R}}^nS) = O(c^n),\] 
for some $0 < c < 1$.

\end{theorem}

This will imply the following linearization result for GIETs. 

\begin{theorem}[Linearization]
\label{thm: linearization}
Let $T = (\lambda, \pi)  \in \Delta_{\A} \times \mathfrak{G}^0_d$ verifying Theorem \ref{thm: affine_shadow} and let $f$ and $S$ as in Theorem \ref{thm: affine_shadow}. Then $f$ and $S$ are $C^1$ conjugated (as GIETs). 


\end{theorem}


Applying the previous results in the piecewise smooth circle homeomorphism setting, we will prove the following.

 \begin{theorem}
 \label{thm: circle_thm}
 Let $d \geq 3$.  For a.e. $T = (\lambda, \pi) \in \Delta_{\A} \times \mathfrak{G}^0_d$, with $\pi$ of rotation type, any two break-equivalent piecewise smooth circle homeomorphisms $f, g \in P^3_{d - 1}(\T)$ verifying
 \[\mathcal{N}(f) = 0 = \mathcal{N}(g), \qquad \gamma(f) = \gamma(T) = \gamma(g),\]
are $C^1$ conjugated as circle maps.
 
 \end{theorem}

\section{Proof of the main results}
\label{sc: outline}

This section contains the proofs of our main results. In Sections \ref{sc:proof_rigidity} and \ref{sc:proof_circle}, we deduce Theorems \ref{thm: rigidity}, \ref{thm: circle_thm} from Theorems \ref{thm: affine_shadow}, \ref{thm: linearization}, respectively.  For the sake of clarity of exposition, we split the proofs of Theorems \ref{thm: affine_shadow} and \ref{thm: linearization}, which are given in Sections \ref{sc:proof_shadow} and \ref{sc:proof_linearization}, into several propositions, whose proof we defer to Sections \ref{sc: existenceofaffineshadow} and \ref{sc: conjugacy}, respectively.

\subsection{Proof of Theorem \ref{thm: rigidity} }
\label{sc:proof_rigidity}

We will deduce Theorem \ref{thm: rigidity} from Theorems \ref{thm: affine_shadow}, \ref{thm: linearization}.  For this, we need to show that the affine models of $f$ and $g$ given by  Theorem \ref{thm: linearization} coincide. 

\begin{lemma}
\label{lem: same_model}
Let $T = (\lambda, \pi)  \in \Delta_{\A} \times \mathfrak{G}^0_d$ verifying Theorem \ref{thm: affine_shadow} and let $f$, $g$  be two boundary-equivalent GIETs of class $C^2$ with $\mathcal{N}(f)=\mathcal{N}(g)=0$  and  $\gamma(f) = \gamma(T) = \gamma(g)$ verifying the EC Condition. Then $\omega_f = \omega_g$, where $\omega_f$ and $\omega_g$ are the vectors given by Theorem \ref{thm: affine_shadow} when applied to $f$ and $g$, respectively.
\end{lemma}
\begin{proof}
Let $S_f$, $S_g$ be AIETs with $\gamma(S_f) = \gamma(T) = \gamma(S_g)$ and log-slopes $\omega_f$ and $\omega_g$, respectively. By Theorem \ref{thm: linearization}, $f$ (resp. $g$) is $C^1$ conjugated, as GIET, to $S_f$ (resp. $S_g$). Hence
\[ \mathcal{B}(S_f) = \mathcal{B}(f) = \mathcal{B}(g) = \mathcal{B}(S_g).\]
Therefore, $\omega_f = \omega_g$ by Proposition \ref{prop: boundary_AIETs}. 
\end{proof}

\begin{proof}[Proof of Theorem \ref{thm: rigidity}]
The result follows directly from Theorems \ref{thm: affine_shadow}, \ref{thm: linearization}, and Lemma \ref{lem: same_model}.
\end{proof}

\subsection{Proof of Theorem \ref{thm: affine_shadow}}
\label{sc:proof_shadow}
Throughout this section, we fix $T = (\tau, \lambda, \pi) \in X$ as in Proposition \ref{prop: generic_condition} and $f$ as in Theorem \ref{thm: affine_shadow}, where $X$ denotes the domain of the acceleration $\mathcal{R}: X \to X$ defined in Section \ref{sc: defR}. We shall see that under these assumptions, the conclusions of the theorem hold.

We start by showing the existence of a vector that `shadows,' with respect to the accelerated height cocycle $Q$ associated with our acceleration $\mathcal{R}$, the logarithm of the mean non-linearity of subsequent renormalizations $\mathcal{R}^n f$ of $f$. 

\begin{proposition}
\label{prop: bound_slopes}
There exists $\omega \in E^{c}(\tau, \lambda, \pi)$ such that
\[ |\omega^n - L^n| = O(c^n),\]
for some $0 < c < 1$, where 
\[\omega^n = \HC{0}{n}(\tau, \lambda, \pi)\omega,\] and $L^n = (L^n_\alpha)_{\alpha \in \A}$ is given by
\begin{equation}
\label{eq: avg_logslope}
L^n_\alpha = \ln  \left( \frac{1}{|I^n_\alpha(f)|} \int_{I^n_\alpha(f)} Df^{q^n_\alpha}(s)ds \right).
\end{equation}
\end{proposition}

Using this vector, we construct AIETs whose orbit under $\mathcal{R}$ shadows that of $f$. 

\begin{proposition}
\label{prop: bound_intervals}
Let $\omega$ as in Proposition \ref{prop: bound_slopes} and $S \in \textup{Aff}(T, \omega)$. Then
\[ \max_{\alpha \in \A} \left| \frac{|I^n_\alpha(f)|}{|I^n(f)|} -  \frac{|I^n_\alpha(S)|}{|I^n(S)|} \right| = O(c^n),\] 
for some $0 < c < 1$.
\end{proposition}


\begin{proposition}
\label{prop: bound_image_intervals}
Let $\omega$ as in Proposition \ref{prop: bound_slopes} and $S \in \textup{Aff}(T, \omega)$. Then
\[ \max_{\alpha \in \A} \left| \frac{|\mathcal{R}^n f(I^n_\alpha(f))|}{|I^n(f)|} -  \frac{|\mathcal{R}^nS(I^n_\alpha(S))|}{|I^n(S)|} \right| = O(c^n),\] 
for some $0 < c < 1$.
\end{proposition}


Assuming the propositions above, we can now prove Theorem \ref{thm: affine_shadow}.

\begin{proof}[Proof of Theorem \ref{thm: affine_shadow}] Since the set $X$ is of full measure, it is clear that for almost every $(\lambda,\pi)$ there exists $\tau$ such that $(\tau,\lambda,\pi)\in X$ verifies Proposition \ref{prop: generic_condition}. 

Let us fix such $(\tau,\lambda,\pi)\in X$  verifying Proposition \ref{prop: generic_condition}, and let $\omega^* \in E^c(\tau,\lambda, \pi)$ be the vector given by Proposition \ref{prop: bound_slopes} when applied to $f$. Denote $\omega_f = \pi_{\Kpi}(\omega^*).$

Let $\omega \in E^c(\tau,\lambda, \pi)$ such that $\pi_{\Kpi}(\omega) = \omega_f$ and $S \in \textup{Aff}(T, \omega)$.  Notice that $\omega - \omega^* \in E^s(\tau,\lambda, \pi)$, since
 $$\pi_{\Kpi}\mid_{E^c(\tau,\lambda,\pi)}: E^c(\tau,\lambda,\pi)\to \Kpi$$ is an isomorphism. That the previous application is, in fact, in isomorphism follows from the duality of the height and length cocycle (see \cite{zorich_deviation_1997}), which implies that $\orth{\Kpi} \cap E_c(\tau,\lambda,\pi) = \{0\}.$ 

Therefore, $\omega$ verifies the conclusions of Proposition \ref{prop: bound_slopes}, and $S$ verifies Propositions \ref{prop: bound_intervals} and \ref{prop: bound_image_intervals}.  By Proposition \ref{prop: bound_slopes}, \ref{prop: bound_intervals}, \ref{prop: bound_image_intervals} and the EC Condition,
\[ \textup{d}_{C^1}({\tilde{\mathcal R}^nf}, {\tilde{\mathcal R}^nS}) = o(c^n),\]
for some $0 < c < 1$. Finally, it follows from Proposition \ref{prop: boundary_properties}, the  EC Condition and the previous equation that
\[ \mathcal{B}(f) = \mathcal{B}({\mathcal R^nf}) = \mathcal{B}({\mathcal R^nS}) + o(c^n) = \mathcal{B}(S) + o(c^n),\]
for some $0 < c < 1$. Therefore $\mathcal{B}(f) = \mathcal{B}(S)$.

\end{proof}


\subsection{Proof of Theorem \ref{thm: linearization}}
\label{sc:proof_linearization}
Recall that for a typical IET $T$, any GIET $f$ with $\gamma(f) = \gamma(T)$ is semi-conjugated to $T$. Hence, for $f$ and $S$ as in Theorem \ref{thm: affine_shadow}, these two maps are semi-conjugated. 
That is, $S \circ h = h\circ f$ for some continuous non-decreasing surjective function $h: [0,1) \to [0, 1)$.  To prove Theorem \ref{thm: linearization}, we will show that this semi-conjugacy is actually a $C^1$ diffeomorphism. 
 
We start by showing that $h$ is a conjugacy between $f$ and $S$. To achieve this, we will use the following result due to C. Ulcigrai and the second author \cite{trujillo_affine_2022}, which states that in the AIET setting, this semi-conjugacy is actually a conjugacy as long as the log-slope vector belongs to the central-stable space of $T$ for the Zorich height cocycle. Let us point out that this was already known for log-slope vectors in the stable space associated with the Zorich height cocycle (see \cite[Theorem 1]{cobo_piece-wise_2002}).

\begin{theorem}[Theorem 1 in \cite{trujillo_affine_2022}]
\label{thm: affconj}
	For almost every IET $T=(\lambda,\pi)$, if $\omega\in E^{cs}(\lambda,\pi)$ then any $S\in\textup{Aff}(T,\omega)$ is $C^0$-conjugated to $T$. 
\end{theorem}

The previous result implies that for $f$ and $S$ as in Theorem \ref{thm: affine_shadow}, the AIET $S$ is topologically conjugated to an IET $T$, in particular, $S$ has no wandering intervals. Combining this with Theorem \ref{thm: affine_shadow}, we shall see in Lemma \ref{lm: conjisLandcont} that $f$ cannot have wandering intervals, and thus it is also conjugated to $T$. Hence, $f$ and $S$ are topologically conjugated, which means that the semi-conjugacy $h$ is actually a conjugacy between these two maps. 
 
Notice that if $h$ was of class $C^1$, by taking derivatives and logarithm in $S \circ h = f \circ h$, we obtain
\[ (\log DS )\circ h - \ln Df = (\log Dh )\circ f - \log Dh. \]
This motivates us to consider the following cohomological equation
\begin{equation}
\label{eq: cohomological_equation}
(\log DS )\circ h - \ln Df = \psi \circ f - \psi.
\end{equation}
Using an appropriate Diophantine condition and Gottschalk-Hedlund Theorem, we prove the following.

\begin{proposition}
\label{prop: cohomological_equation}
Let $f,$ $S$ as in Theorem \ref{thm: affine_shadow} and assume that $f$ is semi-conjugated to $S$ via $h$. Then, there exists a continuous solution $\psi: [0, 1] \to \R$ of (\ref{eq: cohomological_equation}). 

Moreover, if $f$ and $S$ are break-equivalent circle homeomorphisms, $\psi$ can be taken as a continuous function on the circle. 
\end{proposition}

\begin{proposition}
\label{prop: regular_conjugacy}
Let $f,$ $S$ as in Theorem \ref{thm: affine_shadow} and assume that $f$ is semi-conjugated to $S$ via $h$.  Suppose $\psi: [0, 1] \to \R$ is a continuous solution of (\ref{eq: cohomological_equation}). Then,  $h$ is of class $C^1$ and there exists a constant $C$ such that
\[ Dh = Ce^{\psi}.\]
\end{proposition}

\begin{proof}[Proof of Theorem \ref{thm: linearization}]
Theorem \ref{thm: linearization} is a direct consequence of Propositions \ref{prop: cohomological_equation} and \ref{prop: regular_conjugacy}.
\end{proof}

\subsection{Proof of Theorem \ref{thm: circle_thm}}
\label{sc:proof_circle}

The proof of Theorem \ref{thm: circle_thm} makes use of Theorems  \ref{thm: affine_shadow} and \ref{thm: linearization} but requires special considerations on the conjugating map. This is due to the fact that boundary-equivalence, which is a necessary condition for $C^1$ conjugacy of GIETs, does not imply break-equivalence, which is a necessary condition for $C^1$ conjugacy of circle homeomorphisms. 

Thus, although by Theorems \ref{thm: affine_shadow} and \ref{thm: linearization} any function $f$ in the statement of Theorem \ref{thm: circle_thm} will be $C^1$-conjugated, as GIET, to a $1$ dimensional family of AIETs parametrized by the log-slope vector (notice that for an Oseledet's generic IET $(\lambda, \pi)$ of rotation type the space $E^s(\lambda, \pi)$ has dimension $1$), it will be $C^1$ conjugated, as a circle homeomorphism, only to AIETs having a precise log-slope vector (a unique vector in the $1$-dimensional affine subspace of log-slopes given by Theorem \ref{thm: affine_shadow}). 

\begin{proof}[Proof of Theorem \ref{thm: circle_thm}] 
Let $(\lambda, \pi)$ as in Theorem \ref{thm: rigidity}, with $\pi$ of rotation type, and let $f$, $g$ as in the statement of Theorem \ref{thm: circle_thm}. Denote by $F$ and $G$ the canonical identification of $f$ and $g$ as GIETs introduced in Section \ref{sc: p-homeomorphisms}, namely $F = \varphi \circ f \circ \varphi^{-1}$ and $G = \varphi \circ g \circ \varphi^{-1}$, with $\varphi: [0, 1] \to \T$ given by \eqref{eq: parametrization}.

By Theorems \ref{thm: affine_shadow}, \ref{thm: linearization} and  Lemma \ref{lem: same_model}, there exist $\omega  \in E^c(\tau,\lambda, \pi)$ and $L  \in \textup{Aff}(T, \omega)$ such that $F$ and $G$ are $C^1$-conjugated to $L$ as GIETs. Denote by $l$ the circle homeomorphism induced by $L$ with respect to the parametrization $\varphi$, namely, $l = \varphi^{-1} \circ L \circ \varphi$. 

Notice that, although $F$ and $G$ are $C^1$-conjugated to $L$ (in particular boundary-equivalent to $L$) the maps $f$ and $g$ are not necessarily break-equivalent to $l$. In fact, assuming WLOG that $\pi$ verifies \eqref{eq: rotation_type} for some $k \in \{0, \dots, d - 2\}$,  we have
\begin{gather*}
BP(l) \subseteq \{ \varphi(u_j(L)) \mid 0 \leq j < d\}, \\
BP(f) = \{ \varphi(u_j(F)) \mid 0 \leq j < d; \, j \neq d - k - 1\}, \\
BP(g) =  \{ \varphi(u_j(G)) \mid 0 \leq j < d; \, j \neq d - k - 1\}.
\end{gather*}
Note that the conjugacy between $f$ (resp. $g$) and $l$, which we obtain as the induced circle homeomorphism associated with the conjugacy between $F$ (resp. $G$) and $L$,  identifies the points $\varphi(u_j(F))$ (resp. $\varphi(u_j(G))$ with $\varphi(u_j(L))$, for all $0 \leq j \leq d.$ 
Since  $\varphi(u_{d - k - 1}(L))$ could be a break point for $l$, the map $l$ might not be break-equivalent to $f$ and $g$. 

However, it follows from the discussion in Section \ref{sc: p-homeomorphisms} that 
\begin{equation}
\label{eq: equality_jumps}
\sigma_f(\varphi(u_j(F))) =  \sigma_l(\varphi(u_j(L))) = \sigma_g(\varphi(u_j(G))),
\end{equation}
for $0  <  j < d$ with $j \neq d - k - 1$. In particular, the map $l$ has at least $d - 2$ break points, which are mapped by the conjugacy with the map $f$ (resp. $g$) to break points of $f$ (resp. $g$) with the same jump ratio. Thus, if $\varphi(u_{d - k - 1}(L))$ is not a break point for $l$ then the map $l$ is break-equivalent to $f$ and $g$. 

Indeed,  if $\varphi(u_{d - k - 1}(L))$ is not a break point for $l$, then $l$ would have at most $d - 1$ break points. Since $f$ (resp. $g$) has exactly $d - 1$ break points and $BP(l) \subseteq \{ \varphi(u_j(L)) \mid 0 \leq j < d;\,  j \neq d - k - 1\}$, it follows from \eqref{eq: product_jump_ratios}  and  \eqref{eq: equality_jumps} that $l$ has a break point at $\varphi(u_0(L))$ with jump ratio  $\sigma_f(\varphi(u_0(F)))$ (resp. $\sigma_g(\varphi(u_0(G)))$. Since the conjugacy between $f$ (resp. $g$) and $l$ is sending break points to break points,  these maps would be break-equivalent.

In general $\varphi(u_{d - k - 1}(L))$ will be a break point for $l$. However, we can show the following.

\begin{claim}
There exist a unique $v  \in E^s(\lambda, \pi)$ such that for any $L^* \in \textup{Aff}(T, \omega + v)$, we have $\varphi(u_{d - k - 1}(L^*)) \notin BP(l^*)$, where $l^* = \varphi^{-1} \circ L^* \circ \varphi$ denotes the circle homeomorphism induced by $L^*$.
\end{claim}
\begin{proof}[Proof of the Claim]
Since $\pi$ is of rotation type, it follows from the duality of the heights and lengths cocycle (see \cite[Lemma 3.3]{zorich_deviation_1997}, and \cite[pages 384-385]{cobo_piece-wise_2002}) that  
\begin{equation}
\label{eq: stable_space}
E^s(\lambda, \pi) =  \lambda^\bot \cap \textup{Ker}(\Omega_\pi)^\bot.
\end{equation}
Moreover, since $\pi$ verifies \eqref{eq: rotation_type}, a simple calculation shows that
\[E^s(\lambda, \pi) =   \left\{ v \in \R^\A \left| \begin{array}{l}  \langle v, \lambda \rangle = 0,\\ v_\alpha = v_\beta \text{ if } \pi_0(\alpha), \pi_0(\beta) \leq d - k - 1, \\ v_\alpha = v_\beta \text{ if } \pi_0(\alpha), \pi_0(\beta) > d - k - 1. \end{array}  \right\}  \right..  \]

 Let $v \in E^s(\lambda, \pi)$ and denote $ a = v_{\pi_0^{-1}(1)}$ or $b = v_{\pi_0^{-1}(d)}$. Notice that any coordinate of $v$ is either equal to $a$ or $b$. Since $\langle \lambda, v \rangle = 0$, we have
\[b = ta, \quad \text{ where } \quad t = -\frac{\sum_{\pi_0(\alpha) \leq d - k - 1} \lambda_\alpha }{\sum_{\pi_0(\alpha) >  d - k - 1} \lambda_\alpha }. \]
With these notations, it is easy to show that, for any $L^* \in \textup{Aff}(T, \omega + v)$,
\[ \sigma(\varphi(u_j(L^*))) = \left\{ \begin{array}{lcl} 
\sigma_l(\varphi(u_0(L))) \exp^{a(t - 1)} & \text{ if } & j = 0,\\
\sigma_l(\varphi(u_{d - k - 1}(L))) \exp^{a(1 - t)}  & \text{ if }  & j = d - k -1, \\
\sigma_l(\varphi(u_j(L))) & & \text{ otherwise.} 
 \end{array} \right. \]
Hence, there exists a unique $a \in \R$ (and thus a unique $v \in E^s(\lambda, \pi)$) such that $\sigma_L(\varphi(u_{d - k - 1}(L^*))) = 1$, for any $L^* \in \textup{Aff}(T, \omega + v)$. In particular, $\varphi(u_{d - k - 1}(L^*))$ is not a break point of $l^* = \varphi^{-1} \circ L^* \circ \varphi$.  
\end{proof}

Recall that $E^s(\lambda, \pi)$ is orthogonal to $\textup{Ker}(\Omega_\pi)$ by \eqref{eq: stable_space}. Hence, by definition of $\omega$, any AIET $L^*$ as in the previous claim fulfills the conditions in Theorem \ref{thm: affine_shadow} and therefore, by Theorem \ref{thm: linearization}, is $C^1$ conjugated to $F$ and $G$ as GIETs. As before, the induced circle homeomorphism $l^*$ is topologically conjugated to $f$ and $g$, and $d - 2$ of the respective jump ratios at the break points coincide. Since $\varphi(u_{d - k - 1}(L^*))$ is not a break point for $l^*$, it follows from the previous discussion that $l^*$ is break-equivalent to $f$ and $g$. Finally, by Propositions \ref{prop: cohomological_equation} and \ref{prop: regular_conjugacy}, $l^*$ is $C^1$-conjugated to $f$ and $g$ as circle homeomorphisms. 

\end{proof}


\section{Existence of an affine shadow} 
\label{sc: existenceofaffineshadow}

This section is devoted to the proof of Propositions \ref{prop: bound_slopes}, \ref{prop: bound_intervals}, and \ref{prop: bound_image_intervals}. 

\subsection{Proof of Proposition \ref{prop: bound_slopes}}
\label{sc: affine_shadow}
In the following, we fix $T = (\tau, \lambda, \pi) \in X$ as in Proposition \ref{prop: generic_condition} and $f$ as in Theorem \ref{thm: affine_shadow}.

Let us start by showing that the sequence $(L_n)_{n \in \N}$ behaves as a pseudo-orbit with respect to the heights cocycle.
 
\begin{lemma}
\label{lem: pseudo-orbit}
Then  $|L^{n + 1} - \HC{n}{n + 1} L^n| = O(c^n)$, for some $0 < c < 1$.
\end{lemma}
\begin{proof}
For any $\alpha \in \A$ and $n \in \N$, let $x_\alpha^n \in I_\alpha^n(f)$ such that \[L^n_\alpha = \ln Df^{q^n_\alpha}(x_\alpha^n).\]
Let us fix $n \in \N$. Given $\alpha \in \A$, let
\[b_\alpha^n = \sum_{\beta \in \A} \left(\HC{n}{n+1}\right)_{\alpha \beta}.\]
For any $\alpha \in \A$, we can express $q^{n+1}_\alpha$ uniquely as
\[ q^{n+1}_\alpha = \sum_{i = 1}^{b_\alpha^n} q^{n}_{\delta_i(\alpha)}, \]
for some $\delta_i(\alpha) \in \A$, such that 
\begin{equation}
\label{eq: subtowers1}
 f^{h_i}(I^{n+1}_\alpha(f)) \subset I^n_{\delta_i(\alpha)}(f), \hskip0.5cm \text{ where } \hskip0.5cm  h_i = \sum_{j = 1}^{i - 1} q^{n}_{\delta_j(\alpha)},
\end{equation}
for $i = 1, \dots, b_\alpha^n$. Notice that 
\[ (\HC{n}{n+1}L^n)_\alpha = \sum_{\beta \in \A} \left(\HC{n}{n+1}\right)_{\alpha \beta}L^n_\beta,\]
for any $\alpha \in \A$, and
\[ \# \{1 \leq i \leq b_\alpha^n \mid \delta_i(\alpha) = \beta\}  =  \left(\HC{n}{n+1}\right)_{\alpha \beta},\]
for any $\alpha, \beta \in \A$. A simple calculation shows that
\[  (L^{n+1} - \HC{n}{n+1}L^n)_\alpha = \sum_{\beta \in \A} \sum_{\delta_i(\alpha) = \beta} \ln \dfrac{Df^{q^n_\beta}(f^{h_i}(x^{n+1}_\alpha))}{Df^{q^n_\beta}(x^n_\beta)},\]
for any $\alpha \in \A$.   By (\ref{eq: subtowers1}) and Corollary \ref{cor: exp_convergence}, the previous equation yields to
\[\frac{|L^{n+1} - \HC{n}{n+1}L^n|}{\| \HC{n}{n+1}\|} = O(c^n),\]
for some $0 < c < 1$. The claim now follows from Condition \ref{Roth_type} and the fact that $\| \HC{0}{n}\|$ grows exponentially.
\end{proof}

We shall see that this pseudo-orbit can be shadowed by the iterates of a vector with respect to the height cocycle. For any $n \geq 0$, let us decompose $L^n$ with respect to the Oseledet's splitting at $(\tau^n, \lambda^n, \pi^n)$
as 
\[L^n = L^{n, s} + L^{n, c} + L^{n, u},\]
and define
\[ v_n = \HC{0}{n}^{-1} L^{n,c}.\]
\begin{lemma}
There exists $\omega \in E^c(\tau, \lambda, \pi)$ such that $\lim_{n \to \infty}v_n = \omega$. Moreover 
\[|\omega - v_n| = O(c^n),\] 
for some $0 < c < 1$.
\end{lemma}
\begin{proof}We have
\begin{align*}
|v_{n + 1} - v_n|& = \left| \HC{0}{n + 1}^{-1} L^{n + 1, c} - \HC{0}{n}^{-1}L^{n, c}\right| \\
& = |\HC{0}{n + 1}^{-1}\left(L^{n + 1, c} - \HC{n}{n + 1} L^{n, c}\right)| \\
& \leq \left\| \HC{0}{n + 1}^{-1} \mid_{E^n_c}\right\| \left|L^{n + 1, c} - \HC{n}{n + 1} L^{n, c}\right| \\
& = O(c^n),
\end{align*}
where the last equality follows from Condition \ref{bounded_central} and Lemma \ref{lem: pseudo-orbit}.
\end{proof}

\begin{proof}[Proof of Proposition  \ref{prop: bound_slopes}] 
By the EC Condition, we have
\begin{equation}\label{eq: stab+unstab} |L^{n, s}|, |L^{n, u}| = O(c^n),\end{equation}
for some $0 < c < 1$.  
Moreover, by Condition \ref{bounded_central} and the previous lemmas,
\begin{align*}
|\omega^{n} - L^{n, c}|& = \left| \HC{0}{n} \omega - L^{n, c}\right| \\
& = \left| \HC{0}{n}(\omega - \HC{0}{n}^{-1}L^{n, c})\right| \\
& \leq \left \| \HC{0}{n}\mid_{E^c} \right\| |\omega - v_{n}|\\
& = O(c^n),
\end{align*}
for some $0 < c < 1$ which, together with \eqref{eq: stab+unstab} finishes the proof of Proposition \ref{prop: bound_slopes}.
\end{proof}

\subsection{Proof of Proposition \ref{prop: bound_intervals}}
Let us start by introducing a matrix associated with the action of the Zorich algorithm on the slope vectors of AIETs. This matrix will allow us to describe the change in the length vector of AIETs under the action of the Zorich algorithm. 

Given an IET $T = (\lambda, \pi)$ and $\omega \in \R^\A$, we associate the matrix $U(\lambda, \pi, \omega): \R^\A \to \R^A$ given by
 \[
 (U(\lambda, \pi, \omega))_{\alpha,\beta}=\begin{cases}
1, & \text{ if }\alpha=\beta \text{ and }\beta\neq\pi_{\epsilon}^{-1}(d),\\
\exp(\epsilon\cdot \omega_{\pi^{-1}_{\epsilon}(d)})& \text{ if }\alpha=\beta=\pi_{(1-\epsilon)}^{-1}(d),\\
\exp((1-\epsilon)\cdot  \omega_{\pi^{-1}_{(1-\epsilon)}(d)})& \text{ if }\alpha=\pi_{\epsilon}^{-1}(d)\text{ and }\beta=\pi_{(1-\epsilon)}^{-1}(d),\\
0& \text{ otherwise,}
\end{cases}
\]
where $\epsilon=\epsilon(\lambda,\pi)$ denotes the type of the IET defined by $(\lambda,\pi)$. 

Given $n\in\N$, let 
\[
V_n(\lambda, \pi, \omega):=\prod_{i=k(n)}^{k(n+1)-1} U(\mathcal{RV}^i(\lambda,\pi),A^i(\lambda,\pi)\omega).
\]
where $k(n)$ is such that $\mathcal{RV}^{k(n)}(\lambda,\pi,\omega)=\mathcal R^n(\lambda,\pi,\omega) $. The matrix $V_n$ allows us to describe the change of lengths for AIETs that have log-slope $\omega$ and follow the same finite path in the Rauzy graph as $(\lambda,\pi)$. More precisely, if $S \in \textup{Aff}(T, \omega)$, it follows that 
\begin{equation}\label{eq: mataffine}
\left(|I^n_{\alpha}(S)|\right)_{\alpha\in\A} =V_n \left(|I^{n+1}_\alpha(S)|\right)_{\alpha\in\mathcal A}.
\end{equation}
In particular, for the normalized length vector, we have
\[
\big(|\tilde I^n_{\alpha}(S)|\big)_{\alpha\in\A} =V_n \big(|\tilde I^{n+1}_\alpha(S)|\big)_{\alpha\in\mathcal A}.
\]
For the sake of simplicity we will denote the vectors $\left(|I^n_{\alpha}(S)|\right)_{\alpha\in\A} $ and $\left(|I^n_{\alpha}(f)|\right)_{\alpha\in\A} $ as $I^n(S) $ and $I^n(f)$ respectively. The analogous abuse of notation applies to the normalized vectors as well.

For the remainder of this section, we fix $T = (\tau, \lambda, \pi) \in X$ as in Proposition \ref{prop: generic_condition}, $f$ as in Theorem \ref{thm: affine_shadow}, $\omega$ as in Proposition \ref{prop: bound_slopes} and let $S \in \textup{Aff}(T, \omega)$.

\begin{lemma}
\label{lem: compact}
There exists a compact set $K \subset \Delta_{\A}$ such that $I^n(f), I^n(S) \in K$, for any $n \geq 1$. In particular, there exists  a constant $C_K > 0,$ depending only on $K$, such that 
\begin{equation}
\label{eq: balanced_base_renorm}
\sup_{n \geq 1} \max_{\alpha, \beta \in \A} \max\left\{ \frac{|\tilde I_\alpha^n(f)|}{|\tilde I_\beta^n(f)|}, \frac{|\tilde I_\alpha^n(S)|}{|\tilde I_\beta^n(S)|}\right\} < C_K. 
\end{equation}
\end{lemma}
\begin{proof}
Let $D > 0$ and define
\[ K = \overline{\bigcup_{D^{-1}<A_{\alpha\beta}< D} \tilde{A}(\Delta_{\A})}, \]
where the union is taken over all $d \times d$ positive matrices $A$ whose coefficients $A_{\alpha\beta}$ are uniformly bounded from above and below by $D$ and $D^{-1}$ respectively. It is not difficult to check that $K$ is a compact subset of $\Delta_{\A}$ (see \cite{viana_ergodic_2006}), and that there exists a constant $C_K > 0$ such that 
\[\max_{\alpha, \beta \in \A} \frac{w_\alpha}{w_\beta} < C_K, \]
for any $w\in K$. Thus it is sufficient to show that $\tilde I^n(f), \tilde I^n(S)\in K$, for any $n \geq 1$ if $D$ is sufficiently big.

By \eqref{lengthrelation} we have
\[\tilde I^n(f)\in \widetilde{{A}^{|\gamma|}} (\mathcal R^n f)(\Delta_{\A}), \hskip1cm \tilde I^n(S)\in \widetilde{{A}^{|\gamma|}}({\mathcal R^nS}) (\Delta_{\A}),\]
for any $n \geq 0$. 
Since the log-slope of $S$ belongs to $E^c(\tau,\lambda, \pi)$, by Condition \ref{bounded_central}  there exists a constant $C>0$ such that 
\[C^{-1}<\| D{\mathcal R^nS} \mid_{I^n_\alpha(S)}\| <C \quad\text{for every }n\in\N \text{ and }\alpha\in\mathcal A.\]  
Thus, by Proposition \ref{prop: bound_slopes} and Corollary \ref{cor: exp_convergence}, by enlarging $C$ if necessary, we get
\begin{equation}\label{eq:boundedderivative}
C^{-1}<\| D{\mathcal R^nf} \mid_{I^n_\alpha(f)}\| <C \quad\text{for every }n\in\N \text{ and }\alpha\in\mathcal A.
\end{equation}  
Hence there exists a constant $\tilde C>0$ such that
\[\tilde C^{-1}<A^{|\gamma|}_{\alpha\beta}({\mathcal R^nf}),A^{|\gamma|}_{\alpha\beta}({\mathcal R^nS})<\tilde C \quad\text{for every }n\in\N \text{ and }\alpha,\beta\in\mathcal A.\]
Therefore $\tilde I^n(f), \tilde I^n(S) \in K$, for any $n \geq 1$, if $D$ is sufficiently big.
\end{proof}

\begin{lemma}
\label{lem: lengths_pseudo_orbit}
	For every $n\in\N$ we have 
            \[ \frac{1}{|I^{n}(f)|}\left| I^n(f) - V_n I^{n + 1}(f) \right| = O(c^n),\]
	for some $0<c<1$.
\end{lemma}

\begin{proof}
Fix $n \geq 1$. 
We use the notations from the proof of Lemma \ref{lem: pseudo-orbit}. Namely,  for any $\alpha \in \A$, we denote
\[b_\alpha^n = \sum_{\beta \in \A} \left(\HC{n}{n+1}\right)_{\alpha \beta},\]
and express $q^{n+1}_\alpha$ uniquely as
\[ q^{n+1}_\alpha = \sum_{i = 1}^{b_\alpha^n} q^{n}_{\delta_i(\alpha)}, \]
for some $\delta_i(\alpha) \in \A$, such that 
\begin{equation*}
\label{eq: subtowers}
 f^{h_i(\alpha)}(I^{n+1}_\alpha(f)) \subset I^n_{\delta_i(\alpha)}(f), \hskip0.5cm \text{ where } \hskip0.5cm  h_i (\alpha)= \sum_{j = 0}^{i - 1} q^{n}_{\delta_j(\alpha)},
\end{equation*}
for $i = 1, \dots, b_\alpha^n$. Note that $f_{n+1}^i|_{I^{n+1}_\alpha}=f^{h_i^n(\alpha)}|_{I^{n+1}_\alpha}$. We define $\delta_0(\alpha) = 0 = h_0(\alpha),$ for any $\alpha \in \A$. Notice that 
\[ \# \{1 \leq i \leq b_\alpha^n \mid \delta_i(\alpha) = \beta\}  =  \left(\HCT{n}{n+1}\right)_{\alpha \beta},\]
for any $\alpha, \beta \in \A$. For the sake of simplicity, we will denote $\big(\HCT{n}{n+1}\big)_{\alpha \beta}$ simply by $Q_{\alpha\beta}$.

By \eqref{lengthrelation} we obtain
\begin{equation}
\label{eq: lengths_formula_GIET}
 \big| I^n_\beta(f)\big| =  \sum_{\alpha \in \A}\big|I_\alpha^{n + 1}(f)\big| \sum_{i = 1}^{Q_{\alpha\beta}}  Df^{h^n_{m_i(\alpha,\beta)}(\alpha)}(x_i^n(\alpha,\beta)).
 \end{equation}
Replacing $f$ by $S$ in the previous formula yields
\begin{equation}
\label{eq: lengths_formula_AIET}
\big| I^n_\beta(S)\big| = \sum_{\alpha \in \A}  \big|I_\alpha^{n + 1}(S)\big| \sum_{i = 1}^{Q_{\alpha\beta}} \exp\left( \sum_{j = 0}^{m_i(\alpha, \beta) - 1} \omega^n_{\delta_j(\alpha)}\right),
\end{equation}
which is nothing more than the $\beta$-th coordinate of (\ref{eq: mataffine}) since  
\begin{equation}
\label{eq: twisted_cocycle_coefficients}
(V_{n, n+1})_{\alpha\beta} =   \sum_{i = 1}^{Q_{\alpha\beta}} \exp\left( \sum_{j = 0}^{m_i(\alpha, \beta) - 1} \omega^n_{\delta_j(\alpha)}\right).
\end{equation}

By (\ref{eq: lengths_formula_GIET}) and (\ref{eq: twisted_cocycle_coefficients}),
\begin{equation}
\label{eq: difference}
\begin{aligned}
 &\frac{1}{\big| I^{n}(f)\big|} \big[I^n(f) - V_{n, n+1}I^{n + 1}(f)\big]_\beta = \\
 &    \sum_{\alpha \in \A} \sum_{i = 1}^{Q_{\alpha\beta}}  \frac{\big|I_\alpha^{n + 1}(f)\big|}{\big| I^{n}(f)\big|}\Bigg( Df^{h^n_{m_i(\alpha, \beta)}(\alpha)}(x_i^n(\alpha,\beta)) -  \exp\left( \sum_{j = 0}^{m_i(\alpha, \beta) - 1} \omega^n_{\delta_j(\alpha)}\right) \Bigg).
 \end{aligned}
 \end{equation}
\begin{lemma}
\label{lem: aux_lemma}
	For any $\alpha, \beta \in \A$ and any $1 \leq i \leq Q_{\alpha\beta},$ 
	\[ \frac{\big|I_\alpha^{n + 1}(f)\big|}{\big| I^{n}(f)\big|}\Bigg|Df^{h^n_{m_i(\alpha, \beta)}(\alpha)}(x_i^n(\alpha,\beta)) -  \exp\left( \sum_{j = 0}^{m_i(\alpha, \beta) - 1} \omega^n_{\delta_j(\alpha)}\right) \Bigg| = O(c^n),\]
	for some $0<c<1$.
\end{lemma}

Assuming Lemma \ref{lem: aux_lemma}, the result now follows from equations (\ref{eq: balanced_base_renorm}), (\ref{eq: difference}) as well as Condition  \ref{Roth_type}.
\end{proof}
\begin{proof}[Proof of Lemma \ref{lem: aux_lemma}.] 

For any $\alpha, \beta \in \A$ and any $1 \leq i \leq Q_{\alpha\beta}$, we have
\begin{align*}
\Bigg|  \ln Df^{h^n_{m_i(\alpha, \beta)}(\alpha)} & (x_i^n(\alpha,\beta)) -  \sum_{j = 0}^{m_i(\alpha, \beta) - 1} \omega^n_{\delta_j(\alpha)}\Bigg|  \\
& \leq  \sum_{j = 0}^{m_i(\alpha, \beta) - 1}  \left| \ln Df^{q^n_{\delta_j(\alpha)}}\big( f^{h^n_j(\alpha)}(x_i^n(\alpha,\beta)) \big) - \omega^n_{\delta_j(\alpha)}\right| \\
& \leq  \sum_{j = 0}^{m_i(\alpha, \beta) - 1} \big( \big| L^n_{\delta_j(\alpha)} - \omega^n_{\delta_j(\alpha)} \big| + O(c^n) \big) \\
&  = O(c^n),
\end{align*}
for some $0 < c < 1$, where the third line in the above expression follows from the Corollary \ref{cor: exp_convergence} and the fact that $m_i(\alpha,\beta)\le\|Q_{n,n+1}\|$, while the last line follows from Proposition \ref{prop: bound_slopes} as well as Condition  \ref{Roth_type}. 

By the previous equation, 
\begin{equation}
\label{eq: slope_ratio}
\left| \frac{Df^{h^n_{m_i(\alpha, \beta)}(\alpha)}(x_i^n(\alpha,\beta))}{\exp\left(\sum_{j = 0}^{m_i(\alpha, \beta) - 1} \omega^n_{\delta_j(\alpha)} \right)} - 1 \right| = O(c^n).
\end{equation}
The following observation will be useful in the proof. 
\begin{claim}
There exists $C > 0$ such that 
\[ \max_{\alpha \in \A} \frac{\big| I^{n + 1}_\alpha(f)\big|}{\big|I^{n}(f)\big|}\frac{\big|I^{n}(S)\big|}{\big| I^{n + 1}_\alpha(S)\big|} \leq C\|Q_{n, n+1}\|,\]
for any $n \geq 1$.
\end{claim}
\begin{proof}
Fix $\beta \in \A$. By Lemma \ref{lem: compact}, 
\[  \frac{\big| I^{n + 1}_\beta(f)\big|}{\big|I^{n}(f)\big|}\frac{\big|I^{n}(S)\big|}{\big| I^{n + 1}_\beta(S)\big|}  \leq dC_K  \frac{\big| I^{n + 1}_\beta(f)\big|}{\big|I^{n}_\beta(f)\big|}\frac{\big|I^{n}_\beta(S)\big|}{\big| I^{n + 1}_\beta(S)\big|}.\]
By  (\ref{eq: balanced_base_renorm}), (\ref{eq: lengths_formula_GIET}), (\ref{eq: lengths_formula_AIET}) and (\ref{eq: slope_ratio}),
\begin{align*}
 \frac{\big| I^{n + 1}_\beta(f)\big|}{\big|I^{n}_\beta(f)\big|}\frac{\big|I^{n}_\beta(S)\big|}{\big| I^{n + 1}_\beta(S)\big|} & \leq C_K \sum_{\alpha \in \A} \sum_{i = 1}^{Q_{\alpha\beta}}  \frac{\big| I^{n + 1}_\beta(f)\big|}{\big|I^{n}_\beta(f)\big|} \exp\left( \sum_{j = 0}^{m_i(\alpha, \beta) - 1} \omega^n_{\delta_j(\alpha)}\right)
\\
& \leq C_K^2  \sum_{\alpha \in \A} \sum_{i = 1}^{Q_{\alpha\beta}}  \frac{\exp\left( \sum_{j = 0}^{m_i(\alpha, \beta) - 1} \omega^n_{\delta_j(\alpha)}\right)}{Df^{h^n_{m_i(\alpha, \beta)}(\alpha)}(x_i^n(\alpha,\beta))} \\
&  \leq C_K^2  \sum_{\alpha \in \A} \sum_{i = 1}^{Q_{\alpha\beta}} (1 + O(c^n))\\
&
\le C\|Q_{n,n+1}\|,
\end{align*}
for some $C>0$.
\end{proof}

Finally, by (\ref{eq: lengths_formula_AIET}), Condition  \ref{Roth_type} and the previous claim, there exists $C > 0$ such that
\begin{align*}
& \frac{\big|I_\alpha^{n + 1}(f)\big|}{\big| I^{n}(f)\big|} \Bigg|Df^{h^n_{m_i(\alpha, \beta)}(\alpha)}(x_i^n(\alpha,\beta)) -  \exp\left( \sum_{j = 0}^{m_i(\alpha, \beta) - 1} \omega^n_{\delta_j(\alpha)}\right) \Bigg| \\
& =   \left| \frac{Df^{h^n_{m_i(\alpha, \beta)}(\alpha)}(x_i^n(\alpha,\beta))}{\exp\left(\sum_{j = 0}^{m_i(\alpha, \beta) - 1} \omega^n_{\delta_j(\alpha)} \right)} - 1 \right|  \frac{\big|I_\alpha^{n + 1}(f)\big|}{\big| I^{n}(f)\big|}\frac{\big|I^{n + 1}_\alpha(S)\big|}{\big|I^{n + 1}_\alpha(S)\big|}\exp\left(\sum_{j = 0}^{m_i(\alpha, \beta) - 1} \omega^n_{\delta_j(\alpha)} \right) \\
& \leq   \left| \frac{Df^{h^n_{m_i(\alpha, \beta)}(\alpha)}(x_i^n(\alpha,\beta))}{\exp\left(\sum_{j = 0}^{m_i(\alpha, \beta) - 1} \omega^n_{\delta_j(\alpha)} \right)} - 1 \right|  \frac{\big|I_\alpha^{n + 1}(f)\big|}{\big| I^{n}(f)\big|}\frac{\big|I^{n}(S)\big|}{\big|I^{n + 1}_\alpha(S)\big|} \\
& \leq  C \left| \frac{Df^{h^n_{m_i(\alpha, \beta)}(\alpha)}(x_i^n(\alpha,\beta))}{\exp\left(\sum_{j = 0}^{m_i(\alpha, \beta) - 1} \omega^n_{\delta_j(\alpha)} \right)} - 1 \right|    \|Q_{n, n+1}\| \\
& = O(c^n),
\end{align*}
for some $0 < c < 1$.

\end{proof}

As a corollary, we have the following.

\begin{corollary}\label{lem:expdecay}
	For every $n\in\N$ we have 
	\[
	\tilde I^n(f)=\tilde V_n\tilde I^{n+1}(f)+O(c^n),
	\]
	for some $0<c<1$.
\end{corollary}
\begin{proof}
	By Lemma \ref{lem: lengths_pseudo_orbit} we have 
	\[
	|V_n I^{n+1}(f)|=(1+O(c^n))|I^{n}(f)|.
	\]
	Thus, again by Lemma \ref{lem: lengths_pseudo_orbit}, we obtain
	\[
	\begin{split}
	\tilde V_nI^{n+1}(f)&=\frac{V_nI^{n+1}(f)}{|V_nI^{n+1}(f)|}=\frac{V_nI^{n+1}(f)}{|I^{n}(f)|}\cdot\frac{|I^{n}(f)|}{|V_nI^{n+1}(f)|}\\
	&=\frac{I^{n}(f)}{|I^{n}(f)|}(1+O(c^n))=[I^n(f)](1+O(c^n))\\
	&=\tilde I^n(f)+O(c^n).
	\end{split}
	\]
\end{proof}
To formulate the next result, let us recall the classical notion of \emph{Hilbert projective metric} $d_p$ on $\Delta_{\A}$, given by
\[
d_p(v,w)=\log \sup\left\{\frac{v_\alpha}{w_\alpha}\frac{w_\beta}{v_\beta};\ \alpha,\beta\in\mathcal A\right\},
\]
for any $v, w \in \Delta_{\A}$. This metric is equivalent to the Euclidean metric in any compact subset of $\Delta_{\A}$.

\begin{lemma}\label{lem: contraction}
There exists $0<\kappa<1$ such that, for every $n \in \N$, $V_n$ is a contraction in the projective metric with $\kappa$ as contraction constant.
\end{lemma}
\begin{proof}
	Note that $V_n=A^{|\gamma|}({\mathcal R^nS})\cdot B({\mathcal R^nS})$, where $B({\mathcal R^nS})$ is some non-negative matrix. It is a classical fact (see, e.g., (92) in \cite{viana_ergodic_2006}) that multiplication by a non-negative matrix is non-expanding in the projective metric. It is thus enough to show that multiplication by $A^{|\gamma|}({\mathcal R^nS})$ is a contraction with a constant that does not depend on $n$.
	
First note that $A^{|\gamma|}({\mathcal R^nS})$ is positive, since $A^{|\gamma|}$ is positive.
By Proposition 26.3 in \cite{viana_ergodic_2006}, it is enough to show that the entries of $V_n$ are uniformly bounded from below and above. However in view of Condition \ref{bounded_central} (and the fact that $\omega\in E^c(\tau,\lambda,\pi)$) we have that 
\[
\max_{\alpha\in\mathcal A}|\omega_\alpha^n|<K,
\]
for some $1\le K<\infty$ uniformly in $n\in\N$. In particular, we get that for every $n\in\N$ we have
\[
\max_{\alpha,\beta\in\mathcal A}A^{|\gamma|}_{\alpha,\beta}({\mathcal R^nS})\le (\#\mathcal A)^{|\gamma|}e^K\cdot e^{2K}\cdot e^{4K}\cdot\ldots\cdot e^{2^{|\gamma|}K},
\]
and
\[
\min_{\alpha,\beta\in\mathcal A}A^{|\gamma|}_{\alpha,\beta}({\mathcal R^nS})\ge e^{-K}\cdot e^{-2K}\cdot e^{-4K}\cdot\ldots\cdot e^{-2^{|\gamma|}K}.
\]
Since both bounds do not depend on $n$, this concludes the proof of the lemma.
\end{proof}
The ideas used for the following proof come from \cite{cunha_rigidity_2014}.
\begin{proof}[Proof of Proposition \ref{prop: bound_intervals}]
	For $n<m$ let us denote 
	\[
	\tilde V_n^m:=\tilde V_n\circ\ldots\circ \tilde V_m.
	\]
	By Corollary \ref{lem:expdecay} and Lemma \ref{lem: contraction} with Lemma \ref{lem: compact}, we have
	\begin{equation*}
	\begin{split}
	d_p(\tilde I^{n}(f)&,\tilde V_{n}^{2n-1} \tilde I^{2n}(f))\\&\le\sum_{i=0}^{n-1} d_p(\tilde V_{n}^{2n-i-1}\tilde I^{2n-i}(f),\tilde V_{n}^{2n-(i+1)-1}\tilde I^{2n-i-1}(f))\\
	&\le \sum_{i=0}^{n-1}\kappa^{n-i-1}O(c^{2n-(i+1)})\le{(1-\kappa)^{-1}}O(c^{n})=O(c^{n}).
	\end{split}
	\end{equation*}
	Thus, by \eqref{eq: mataffine},
	\begin{equation}
	\begin{split}
	d_p(\tilde I^{n}(f),\tilde I^{n}(S))&\le d_p(\tilde I^{n}(f),\tilde V_{n}^{2n-1} \tilde I^{2n}(f))\\ &\quad +
	d_p(\tilde V_{n}^{2n-1} \tilde I^{2n}(f),\tilde V_{n}^{2n-1}\tilde I^{2n}(S))\\ &\le O(c^{n})+\kappa^n\cdot D,
	\end{split}
	\end{equation}
	where $D$ is the diameter in $d_p$ of the compact set given by Lemma \ref{lem: compact}, which contains both $\tilde I^{2n}(f)$ and $\tilde	I^{2n}(S)$.
	Thus, we get 
	\[
	d_p(\tilde I^{n}(f),\tilde I^{n}(S))\le O(c^{n})+O(\kappa^{n}),
	\]
	which yields
	\[
	d_p(\tilde I^n(f),\tilde I^n(S))=O(\max\{c,\kappa\}^n).
	\]
	
	\end{proof}

\subsection{Proof of Proposition \ref{prop: bound_image_intervals}}
We use the results of previous sections, in particular, Proposition \ref{prop: bound_slopes} and Proposition \ref{prop: bound_intervals} to prove Proposition \ref{prop: bound_image_intervals}.

\begin{proof}[Proof of Proposition \ref{prop: bound_image_intervals}] 
Note first that 
\[
\frac{|\mathcal{R}^nS(I^n_\alpha(S))|}{|I^n(S)|}=\frac{e^{\omega^n_{\alpha}}|I^n_\alpha(S)|}{|I^n(S)|}.
\]
Moreover, 
\[
\frac{|\mathcal{R}^n f(I^n_\alpha(f))|}{|I^n(f)|}=
\frac{e^{L_\alpha^n} |I^n_\alpha(f)|}{|I^n(f)|}.
\]
Thus, we obtain
\[
\begin{split}
\left| \frac{|\mathcal{R}^n f(I^n_\alpha(f))|}{|I^n(f)|} -  \frac{|\mathcal{R}^nS(I^n_\alpha(S))|}{|I^n(S)|} \right| &=
\left| \frac{e^{L_\alpha^n} |I^n_\alpha(f)|}{|I^n(f)|} -  \frac{e^{\omega^n_{\alpha}}|I^n_\alpha(S)|}{|I^n(S)|} \right|\\
\le\left| \frac{e^{L_\alpha^n} |I^n_\alpha(f)|}{|I^n(f)|} -  \frac{e^{L_\alpha^n}|I^n_\alpha(S)|}{|I^n(S)|} \right|&+\left| \frac{e^{L_\alpha^n} |I^n_\alpha(S)|}{|I^n(S)|} -  \frac{e^{\omega^n_{\alpha}}|I^n_\alpha(S)|}{|I^n(S)|} \right|.
\end{split}
\]
By Proposition \ref{prop: bound_intervals}, 
\[
\left| \frac{e^{L_\alpha^n} |I^n_\alpha(f)|}{|I^n(f)|} -  \frac{e^{L_\alpha^n}|I^n_\alpha(S)|}{|I^n(S)|} \right|=e^{L_\alpha^n}O(c^n),
\]
and by \ref{prop: bound_slopes},
\[
\left| \frac{e^{L_\alpha^n} |I^n_\alpha(S)|}{|I^n(S)|} -  \frac{e^{\omega^n_{\alpha}}|I^n_\alpha(S)|}{|I^n(S)|} \right|=e^{L_\alpha^n}O(c^n),
\]
for some $0 < c < 1$. The result now follows from \eqref{eq:boundedderivative}.
\end{proof}

\section{Smoothness of the conjugacy map}
\label{sc: conjugacy}

This section is devoted to the proof of Propositions \ref{prop: cohomological_equation} and \ref{prop: regular_conjugacy}.  The following consequences of Theorem \ref{thm: affine_shadow} will be helpful. 

\begin{lemma}
\label{lem: interval_bounds}
Let $f$, $S$, as in Theorem \ref{thm: affine_shadow}. Then
\begin{enumerate}
\item  For any $\alpha \in \A$,  $\frac{|I^n(f)|}{|I^n(S)|} \frac{|I^n_\alpha(S)|}{|I^n_\alpha(f)|} = 1 + O(c^n),$ for some $0 < c < 1$.
\item $\frac{|I^n(f)|}{|I^n(S)|} \frac{|I^{n-1}(S)|}{|I^{n-1}(f)|} = 1 + O(c^n),$ for some $0 < c < 1$.
\item For any $\alpha \in \A$,   $\frac{|I^n_\alpha(f)|}{|I^n_\alpha(S)|}, \frac{|{\mathcal R^nf}(I^n_\alpha(f))|}{|{\mathcal R^nS}(I^n_\alpha(S))|}, \frac{|I^n(f)|}{|I^n(S)|}  \to C$ for some $C > 0$.
\end{enumerate}
\end{lemma}
\begin{proof}
The first assertion follows from Lemma \ref{lem: compact} and Proposition \ref{prop: bound_intervals}.  For any $n \geq 1$ we have
\begin{align*}
\frac{|I^n(f)|}{|I^n(S)|} \frac{|I^{n-1}(S)|}{|I^{n-1}(f)|}  & = \frac{\left| \HCT{n - 1}{n} \tilde I^n(S)\right|}{\left| \HCT{n - 1}{n} \tilde I^n(f)\right|} \\
& \leq 1  +  \frac{\left| \HCT{n - 1}{n} \left( \tilde I^n(S) - \tilde I^n(f) \right)\right|}{\left| \HCT{n - 1}{n} \tilde I^n(f)\right|} \\
& \leq 1 + \big\| \HCT{n - 1}{n}\big\| \left| \tilde I^n(S) - \tilde I^n(f) \right| \\
& = 1 + O(c^n),
\end{align*}
for some $0 < c < 1$, by Condition  \ref{Roth_type} as well as Theorem \ref{thm: affine_shadow}. This proves the second assertion.

Notice that by Theorem \ref{thm: affine_shadow} and the first assertion, to prove the last assertion, it is sufficient to show that $\frac{|I^n(f)|}{|I^n(S)|}$ converges as $n \to \infty$. 

By the second assertion,
\begin{align*}
\frac{|I^n(f)|}{|I^n(S)|} & = \frac{|I^0(f)|}{|I^0(S)|}\prod_{i = 1}^{n} \frac{|I^i(f)|}{|I^i(S)|}\frac{|I^{i - 1}(S)|}{|I^{i - 1}(f)|} \\
& =  \frac{|I^0(f)|}{|I^0(S)|}\prod_{i = 1}^{n}(1 + O(c^n)),
\end{align*}
for some $0 < c < 1$. Therefore, $\frac{|I^n(f)|}{|I^n(S)|}$ converges as $n \to \infty$.
\end{proof}

\subsection{Proof of Proposition \ref{prop: cohomological_equation}}

\begin{proof}[Proof of Proposition \ref{prop: cohomological_equation}]
For any bounded $\varphi: [0, 1) \rightarrow \R$, and any $m \geq 0$, denote by $\mathcal \mathcal S_m^f$ the $m$-th Birkhoff sum of $\varphi$ w.r.t. $f$. By decomposing Birkhoff sums of length $m$ into Birkhoff sums along towers obtained via $\mathcal R$ (special Birkhoff sums) we get
\begin{equation}
\label{eq: sum_decomposition}
\big\| \mathcal S_m^f \varphi \big\|_{[0, 1)} \leq \sum_{n \geq 0} \big\|\HC{n}{n + 1}\big\|\max_{\alpha \in \A} \big\|\mathcal S_{q^n_\alpha}^f \varphi \big\|_{I^n_\alpha(f)}.
\end{equation}
For details on this decomposition, we refer the interested reader to \cite[Section 2.2.3]{marmi_cohomological_2005}.
Notice that for any $m \geq 0$, 
\begin{equation}\label{eq:BS} \mathcal S_m^f(\log DS \circ h - \log Df) = \log\frac{DS^m \circ h}{Df^m}.\end{equation}
Pick $\varphi = \log DS \circ h - \log Df$. Notice that $\varphi \in C^0\big(\bigsqcup_{\alpha \in \A} I_\alpha \big)$.  By \eqref{eq:BS} together with Corollary \ref{cor: exp_convergence} and Proposition \ref{prop: bound_slopes} we get
\[\max_{\alpha \in \A} \big\|\mathcal S_{q^n_\alpha}^f \varphi \big\|_{I^n_\alpha(f)}= O(c^n),\]
for some $0 < c < 1$. By the previous equation, (\ref{eq: sum_decomposition}) and Condition  \ref{Roth_type}, for any $\tau > 0$ there exists $C > 0$ such that 
\[\| \mathcal S_m^f \varphi \big\|_{[0, 1)}  \leq C \sum_{n \geq 0}  e^{n(\tau - |\log c|)},\]
for any $m \geq 0$.  Picking $\tau < |\log c|$, it follows that 
\[ \sup_{m \geq 0} \| \mathcal S_m^f \varphi \big\|_{[0, 1)} < +\infty.\]
Hence, by Proposition \ref{prop: cohomological_equation0}, 
there exists $\psi: [0, 1) \to \R$ continuous verifying (\ref{eq: cohomological_equation}).

Furthermore, if $f$ and $S$ are break-equivalent circle homeomorphisms, the function $\varphi$ above is a well-defined continuous function on the circle. Then, considering $\varphi$ and $f$ as continuous functions on $\T$, by Gottschalk-Hedlund's Theorem there exists $\psi: \T \to \R$ continuous  verifying (\ref{eq: cohomological_equation}).
\end{proof}


\subsection{Proof of Proposition \ref{prop: regular_conjugacy}}


\begin{lemma}\label{lm: conjisLandcont}
Let $f$, $S$ be as in Theorem \ref{thm: affine_shadow} and assume that $f$ is semi-conjugated to $S$ via $h$. Then $h$ is a Lipschitz homeomorphism.
\end{lemma}
\begin{proof}
	Let $\alpha \in \A$, $n \geq 0$ and $0 \leq j < q^n_\alpha.$ Since $h(f^j(I^n_\alpha(f))) = S^j(I^n_\alpha(S))$, by the intermediate value theorem,
	\begin{equation*}
	\label{eq: conjugacy_estimates}
	\frac{|h(f^j(I^n_\alpha(f)))|}{|f^j(I^n_\alpha(f))|} = \frac{|{\mathcal R^nS}(I^n_\alpha(S))|}{|{\mathcal R^nf}(I^n_\alpha(f))|} \frac{Df^{q^n_\alpha - j}(x)}{DS^{q^n_\alpha - j}(y)},
	\end{equation*}
	for some $x \in f^j(I^n_\alpha(f))$ and $y \in S^j(I^n_\alpha(S))$. Using Proposition \ref{prop: cohomological_equation}, we can rewrite the previous equation as 
	\begin{equation*}
	\begin{split}& \frac{|h(f^{j}(I^n_{\alpha}(f)))|}{|f^{j}(I^n_{\alpha}(f))|}  = \frac{|{\mathcal R^nS}(I^n_{\alpha}(S))|}{|{\mathcal R^nf}(I^n_{\alpha}(f))|} \frac{DS^{q^n_{\alpha} - j}(h(x))}{DS^{q^n_{\alpha} - j}(y)}\frac{Df^{q^n_{\alpha} - j}(x)}{DS^{q^n_{\alpha} - j}(h(x))} \\
	& = \frac{|{\mathcal R^nS}(I^n_{\alpha}(S))|}{|{\mathcal R^nf}(I^n_{\alpha}(f))|} \frac{DS^{q^n_{\alpha} - j}(h(x))}{DS^{q^n_{\alpha} - j}(y)} e^{-\psi \circ f^{q^n_{\alpha} - j}(x)}e^{\psi(x)}. \end{split}
	\end{equation*}
	Note that the points $y$ and $h(x) $ belong to $S^j(I^n_{\alpha}(S))$, which is a continuity interval of $S^k$ for every $k=0,\ldots,q_{\alpha}^n-j$. This, together with the fact that $S$ is an AIET, yields $ \frac{DS^{q^n_{\alpha} - j}(h(x))}{DS^{q^n_{\alpha} - j}(y)}  = 1.$
	By Lemma \ref{lem: interval_bounds} and Proposition \ref{prop: cohomological_equation} and the previous equation,  
	\begin{equation}
	\label{eq: ratSf}
	 \sup_{n \geq 0} \max_{J \in \mathcal{P}_n(f)} \max\left\{ \frac{|h(J)|}{|J|}, \frac{|J|}{|h(J)|} \right\}< +\infty,
	 \end{equation}
	where $\mathcal P_n(f)=\{f^j(I_\alpha^n(f));\ \alpha\in\mathcal A,\ 0\le j<q^n_\alpha\}$.
	Since any interval  $(a, b) \subset [0, 1)$ can be expressed as a countable union of intervals in $\bigcup_{n \geq 0} \mathcal{P}_n(f),$ it follows that $h$ is a Lipschitz function. 
	
	Since by Theorem \ref{thm: affconj}, $S$ is topologically conjugated to a minimal IET, we have
\[
\lim_{n\to\infty}\max_{\substack{\alpha\in\mathcal A\\ 0\le j<q_{\alpha}^n}}|S^j(I_{\alpha}^{n}(S))|=0.
\]
Hence by \eqref{eq: ratSf} we obtain
\[
\lim_{n\to\infty}\max_{\substack{\alpha\in\mathcal A\\ 0\le j<q_{\alpha}^n}}|f^j(I_{\alpha}^{n}(f)|=0.
\]
Since for every $n\ge 0$, the wandering intervals are always included in the elements of the partition $\mathcal P_n(f)$, the above expression implies that $f$ does not have wandering intervals, which in turn implies that $h$ is a conjugacy. 

	\end{proof}

\begin{proof}[Proof of Proposition \ref{prop: regular_conjugacy}]
By Lemma \ref{lm: conjisLandcont}, the map $h$ is a Lipschitz homeomorphism. It follows that it is differentiable almost everywhere. Let $x_0 \in [0, 1)$ such that $h$ is differentiable. WLOG, we may suppose that $x_0$ is not the endpoint of any interval in  $\bigcup_{n \geq 0} \mathcal{P}_n(f).$ Then, there exist sequences $(\alpha_n)_{n \geq 1}\subset \A^{\N}$ and $(j_n)_{n \geq 1} \subset \N$ such that 
\[ 0 \leq j_n < q^n_{\alpha_n}, \hskip1cm \{x_0\} = \bigcap_{n \geq 1}f^{j_n}(I^n_{\alpha_n}(f)).\]
Therefore
\[ Dh(x_0) = \lim_{n \to \infty}  \frac{|h(f^{j_n}(I^n_{\alpha_n}(f)))|}{|f^{j_n}(I^n_{\alpha_n}(f))|}.\]
By the mean value theorem, there exist sequences $x_n \in f^{j_n}(I^n_{\alpha_n}(f))$ and $y_n \in S^{j_n}(I^n_{\alpha_n}(S))$ such that 
\[   \frac{|h(f^{j_n}(I^n_{\alpha_n}(f)))|}{|f^{j_n}(I^n_{\alpha_n}(f))|} = \frac{|S^{j_n}(I^n_{\alpha_n}(S))|}{|f^{j_n}(I^n_{\alpha_n}(f))|}= \frac{|{\mathcal R^nS}(I^n_{\alpha_n}(S))|}{|{\mathcal R^nf}(I^n_{\alpha_n}(f))|} \frac{Df^{q^n_{\alpha_n} - j_n}(x_n)}{DS^{q^n_{\alpha_n} - j_n}(y_n)}.\]
Using Proposition \ref{prop: cohomological_equation} we can rewrite the previous equation as 
\begin{align*}
& \frac{|h(f^{j_n}(I^n_{\alpha_n}(f)))|}{|f^{j_n}(I^n_{\alpha_n}(f))|}  = \frac{|{\mathcal R^nS}(I^n_{\alpha_n}(S))|}{|{\mathcal R^nf}(I^n_{\alpha_n}(f))|} \frac{DS^{q^n_{\alpha_n} - j_n}(h(x_n))}{DS^{q^n_{\alpha_n} - j_n}(y_n)}\frac{Df^{q^n_{\alpha_n} - j_n}(x_n)}{DS^{q^n_{\alpha_n} - j_n}(h(x_n))} \\
 & = \frac{|{\mathcal R^nS}(I^n_{\alpha_n}(S))|}{|{\mathcal R^nf}(I^n_{\alpha_n}(f))|} \frac{DS^{q^n_{\alpha_n} - j_n}(h(x_n))}{DS^{q^n_{\alpha_n} - j_n}(y_n)} e^{-\psi \circ f^{q^n_{\alpha_n} - j_n}(x_n)}e^{\psi(x_n)}.
\end{align*}
We now analyze each term in the previous equation separately. 

By Lemma \ref{lem: interval_bounds}, $\frac{|{\mathcal R^nS}(I^n_{\alpha_n}(S))|}{|{\mathcal R^nf}(I^n_{\alpha_n}(f))|} \to C,$ for some constant $C > 0$. 

Similarly to the proof of Lemma \ref{lm: conjisLandcont}, the points $y_n$ and $h(x_n) $ belong to $S^{j_n}(I^n_{\alpha_n}(S))$, which is a continuity interval of $S^k$ for every $k=0,\ldots,q_{\alpha_n}^n-j$. Since $S$ is an AIET, we get  $$\frac{DS^{j_n}(S^{-j_n}(h(x_n)))}{DS^{j_n}(S^{-j_n}(y_n))} = 1.$$



Since $ f^{q^n_{\alpha_n} - j_n}(x_n) \to 0$, we have $e^{\psi \circ f^{q^n_{\alpha_n} - j_n}(x_n)} \to e^{\psi(0)}$.

Therefore, since $x_n \to x_0$, we have $Dh(x_0) = Ce^{\psi(0)}e^{\psi(x_0)}$. Moreover, since $h$ is Lipschitz and the derivative of $h$ at almost every point coincides with the continuous function $Ce^{\psi(0)}e^{\psi(\cdot)}$, it follows that $h$ is of class $C^1$

\end{proof}


\bibliographystyle{acm}
\bibliography{Bibliography.bib}
\end{document}